\documentclass[pdflatex,sn-mathphys-num]{sn-jnl}% Math and Physical Sciences Numbered Reference Style
\usepackage{graphicx}%
\usepackage{multirow}%
\usepackage{amsmath,amssymb,amsfonts}%
\usepackage{amsthm}%
\usepackage{mathrsfs}%
\usepackage[title]{appendix}%
\usepackage{xcolor}%
\usepackage{textcomp}%
\usepackage{manyfoot}%
\usepackage{booktabs}%
\usepackage{algorithm}%
\usepackage{algorithmicx}%
\usepackage{algpseudocode}%
\usepackage{listings}%
\usepackage{bbm,pgfplots,subcaption,mathtools,mdframed}
\usepgfplotslibrary{fillbetween}

\theoremstyle{thmstyleone}%
\newtheorem{theorem}{Theorem}%  meant for continuous numbers
\theoremstyle{thmstyletwo}%

\theoremstyle{thmstylethree}%
\newtheorem{definition}{Definition}%

\newcommand{\real}{\mathbb{R}}

\newcommand{\I}{\mathcal{I}}

\newcommand{\mc}[1]{\mathcal{#1}}
\newcommand{\od}{\text{d}}
\newcommand{\pd}{\partial}
\newcommand{\deriv}[2]{\frac{\od #1}{\od #2}}
\newcommand{\nderiv}[1]{\frac{\od}{\od #1}}

\newcommand{\pderiv}[2]{\frac{\pd #1}{\pd #2}}
\newcommand{\npderiv}[1]{\frac{\pd}{\pd #1}}
\newcommand{\nsecpderiv}[1]{\frac{\pd^2}{\pd #1 ^2}}
\newcommand{\secpderiv}[2]{\frac{\pd^2 #1}{\pd #2 ^2}}
\newcommand{\del}{\nabla}

\newcommand{\mean}[1]{\left \langle #1 \right \rangle}

\newcommand{\bg}{\mathbf{g}}
\newcommand{\bu}{\mathbf{u}}
\newcommand{\bp}{\mathbf{p}}

\newcommand{\mb}[1]{\mathbf{#1}}
\newcommand{\mnphi}[1]{\mean{\phi(#1)}}
\newcommand{\dt}{\Delta t}
\newcommand{\inner}[2]{\left(#1,#2\right)}

\newcommand{\norm}[1]{\left \lVert #1 \right \rVert}

\newcommand{\F}{\mathcal{F}}
\newcommand{\half}{\frac{1}{2}}
\newcommand{\iplus}{i+\half}
\newcommand{\iminus}{i-\half}

\newmdtheoremenv{Prob}{Problem}
\newmdtheoremenv{Alg}{Algorithm}

\begin{document}

\title[PDE-constrained optimal control of a
leader-follower opinion formation model]{PDE-constrained optimal control of a
leader-follower opinion formation model}

\author*[1]{\fnm{Bertram} \sur{D\"uring}}\email{bertram.during@warwick.ac.uk}
\author[1]{\fnm{Oliver} \sur{Wright}}\email{oliver.l.j.wright@warwick.ac.uk}

\affil[1]{\orgdiv{Mathematics Institute}, \orgname{University of Warwick}, \orgaddress{\street{Zeeman
  Building}, \city{Coventry}, \postcode{CV4 7AL}, \state{West Midlands},
\country{United Kingdom}}}

\abstract{We consider the PDE-constrained optimal control of a leader-follower
kinetic opinion formation model, with a Fokker-Planck-type
system of partial differential equations as a state constraint.
We derive the Boltzmann-type and
Fokker-Planck-type systems of equations associated with the controlled
leader-follower opinion formation model. In a function space setting we
derive first-order optimality conditions associated with the 
PDE-constrained optimal control problem, yielding an optimality system of coupled
nonlinear partial differential equations.
We employ a gradient-type sweeping algorithm
to numerically attack the
optimality system obtained from the first-order optimality
conditions. We present the results from a finite elements based
simulation for different types of interactions and cost functionals.}

\keywords{opinion formation, leaders, optimal control, Fokker–Planck-type system}

\pacs[MSC Classification]{49M41, 91D30, 35Q93}

\maketitle

\section{Introduction}\label{sec0}

We study the PDE-constrained optimal control of a leader-follower
kinetic opinion formation model from \cite{bib:During:strongleaders},
which features {\em two species}, a species of strong opinion leaders and a species of followers,
building on the works of Fornasier and Solombrino
\cite{fornasier2014mean} and Albi \textit{et al.} \cite{bib:Albi:meancontrol}.
In \cite{bib:Albi:meancontrol} the authors
consider the mean field optimal control of a partial differential
equation of continuity type, as they arise in opinion formation
problems for a large population in the mean-field limit {\em for a single species}.
They derive first-order optimality conditions and
propose a sub-optimal, but efficient numerical optimal control strategy.

The introduction of opinion formation leader-follower models in the literature in
the past two decades \cite{bertotti2008discrete,bib:During:strongleaders,bib:During:inhomogeneous} has naturally
sparked curiosity on how leaders can optimise their behaviour, e.g.\ position themselves
optimally to attract a mass of followers. Another question that can be
asked is how leaders can optimally behave to steer a follower density to
a desired target distribution.
Examples in the real world spring to mind, with political parties,
their leaders and their spin doctors, trying to attract voters,
e.g.\ by pushing certain topics where differentiation from other
parties is easier, or to establish them as more visibly different
alternative to other parties. 

PDE constrained optimal control in an infinite dimensional setting has
been dynamically developing as a research area since the late 1990s
and 2000s years (confer
\cite{hinze2008optimization,troltzsch2010optimal} for an overview and references), moving from prototype nonlinear PDE to
systems of nonlinear PDE, e.g.\ from fluid dynamics, and continues to
thrive. For the nonlinear, nonlocal partial
differential equations arising in opinion formation,
many works have favoured numerically less expensive approaches. One
reason is that the nonlocal terms in these equations make their
numerical solution and even more so their numerical optimisation very
challenging.
A popular class of algorithms is known as model predictive control,
which is a powerful, yet sub-optimal approach to attack the complex
control problem \cite{bib:AlbiPareschiZanella:MPC,wongkaew2015control} by solving iteratively small finite-horizon
optimal control problems. Similar many-particle models and associated
control approaches are present in flocking \cite{borzi2015modeling,bailo2018optimal}, crowd dynamics
\cite{albi2020mathematical,burger2020instantaneous,gong2023crowd},
epidemic dynamics \cite{albi2021control,zanella2023kinetic,bondesan2024kinetic}, opinion formation on graphs \cite{during2024breaking}, traffic flow \cite{herty2007instantaneous} and global optimisation \cite{carrillo2018analytical,pinnau2017consensus}.

In this paper, starting from controlled interactions in
Section~\ref{sec1}, we derive the Boltzmann-type and
Fokker-Planck-type systems of partial differential equations
associated with the controlled
leader-follower opinion formation model in Section~\ref{sec2:PDEs}.
Section~\ref{sec2:OCprobs} is devoted to PDE-constrained optimal control.
In an infinite-dimensional
function space setting we
derive first-order optimality conditions associated with the 
PDE-constrained optimal control problem with the Fokker-Planck-type
system as a state constraint, yielding an optimality system of coupled nonlinear partial differential equations.
In Section~\ref{sec3:Nummeth} we discuss the numerical algorithms. We employ a gradient-type sweeping algorithm
\cite{bib:Burgeretal:OCPMFG} to numerically attack the
optimality system obtained from the first-order optimality
conditions. We present the results from a finite elements based
simulation for different types of interactions and cost functionals in
Section~\ref{sec4:Numexp}.

\section{Multi-agent opinion formation model} \label{sec1}
\label{sec1:interactions}

We consider a large number of interacting individuals, with each
individual's opinion in an interval $\I := [-1, 1]$. The overall
population consists of
two species, an opinion leader species and a follower species.
Similar as in \cite{bib:During:strongleaders}, members of either
species can interact with each other through binary interactions, with
three possible cases: (a) the interactions where a member of the
leader species interacts with another member of the leader species (an
L and L interaction); (b) a member of the leader species interacts
with a member of the follower species (an L and F interaction); and
(c) a member of a follower species interacts with another member of a
follower species (an F and F interaction). We assume that the leader
species is formed by strong opinion leaders, as in \cite{bib:During:strongleaders}, that is
assertive individuals being able to influence followers but withstanding being
influenced themselves. To that end, we have that an individual of the
leader species will have their opinion unchanged by an interaction
with an individual from a follower species, but interact normally with
another leader. Unlike \cite{bib:During:strongleaders} we allow a
control function $u$ to modify the leader interactions. The
interactions are then as follows,
        \begin{itemize}
            \item L-L interactions: 
            \begin{equation} \label{int:LL}
                \begin{split}
                w^*_{Li} &= w_{Li} + \gamma_L P_L(w_{Li},w_{Lj}) (w_{Lj}-w_{Li}) + \eta_L D(w_{Li}) \\
                &\qquad + \frac{\gamma_L}{2} u(w_{Li}, w_{Lj},t) \\
                w_{Lj}^* &= w_{Lj} + \gamma_L P_L(w_{Lj},w_{Li}) (w_{Li}-w_{Lj}) + \tilde{\eta}_L D(w_{Lj})\\
                &\qquad + \frac{\gamma_L}{2} u(w_{Lj}, w_{Li},t)
                \end{split}
            \end{equation}
            \item L-F interactions:
            \begin{equation} \label{int:FL}
                \begin{split}
                    w_{Li}^* &= w_{Li} \\
                    w_{Fj}^* &= w_{Fj} + \gamma_F \Tilde{P}(w_{Fj},w_{Li}) (w_{Li}-w_{Fj}) + \tilde{\eta} D(w_{Fj})
                \end{split}
            \end{equation}
            \item F-F interactions:
            \begin{equation} \label{int:FF}
                \begin{split}
                    w_{Fi}^* &= w_{Fi} + \gamma_F P_F(w_{Fi},w_{Fj}) (w_{Fj}-w_{Fi}) + \eta_F D(w_{Fi}) \\
                    w_{Fj}^* &= w_{Fj} + \gamma_F P_F(w_{Fj},w_{Fi}) (w_{Fi}-w_{Fj}) + \tilde{\eta}_F D(w_{Fj})
                \end{split}
            \end{equation}
        \end{itemize}
        where $w_{Li}^*$, $w_{Fi}^*$ are the post-interaction opinions of the individuals.

    Note here that we add the control term only to the leader
    interaction \eqref{int:LL} which represents the effect of a
    particular strategy that involves changing the leader's apparent
    opinion in order to minimise a given cost functional $J$. This cost functional may be as simple as trying to maximise a follower base -- as is likely in the case of some form of marketing strategy -- or, in a paradigm where multiple leader species exist, to minimise the follower base of the other leader species. Since we are only controlling the leader interactions we expect the action of the control function $u$ on the follower species to be only through the interactions with leader species.
    
    In the terms with the form $\gamma_\nu P_\nu(w,v) (v-w)$ and
    $\gamma_\nu \Tilde{P}(w,v) (v-w)$, we are representing a
    compromise between the opinion of the first and second
    individuals. In the prevalent literature, the $P$ functions in
    these interactions has been taken as either  a bounded confidence
    model -- $P$ functions are taken as localisation functions such as
    a characteristic function or a hyperbolic tangent function
    \cite{bib:Toscani:kineticmodel, bib:During:inhomogeneous,
      bib:During:strongleaders} -- using the assumption that
    individuals will only interact with other individuals with similar
    opinions, or the Sznajd model \cite{sznajd2000opinion}-- $P$
    functions penalise extremal 
    opinions -- where it is assumed that more extreme opinions
    interact little with any other opinions. A more in-depth
    discussion of these choices for compromise functions is given below
    in Section~\ref{sec3:Nummeth}. More realistic choices for
    this compromise function while also taking into account the
    different ``demographics" of the population are proposed in \cite{bib:During:polling,during2024voter}.

    The constants $\gamma_L$ and $\gamma_F$ in these terms govern the
    rate of compromise over time. We note here that we choose $P$ such
    that the range of $P$ is in $[0,1]$ and we choose $\gamma_L$ and
    $\gamma_F$ to be in $[0, \tfrac{1}{2})$ to ensure post-interaction
    opinions remain in $\I$ (see \cite{bib:During:strongleaders}). 
    
    The terms of the form $\eta_\nu D(w)$, represent the influence of other sources, which in our case act randomly. This includes things like, media controlled by outside sources, the individual thinking about their opinion. Here we are making the assumption that these outside influences are unbiased. Here $\eta_L$ and $\eta_F$ represent the random elements. The function $D$ is chosen in such a way as to ensure that the diffusion term is sufficiently small near the boundaries -- and should not be confused with a differential operator. This is to ensure that the post-interaction opinions of the individuals $w^*$, $v^*$ remain within $\I$. A usual choice for this function in the literature \cite{bib:Toscani:kineticmodel, bib:During:inhomogeneous, bib:During:strongleaders} is,
    $$D(x) = (1 - x^2)^\alpha,$$
    for some $\alpha > 0$. We use this choice of $D$ with $\alpha = 2$
    throughout our numerical experiments.

    The random variables $\eta_L$ and $\eta_F$ are chosen such that their probability density function $\Theta$ from the set of probability measures,
    \begin{align*}
        M_{2+\delta} &= \bigg\{ \mathbb{P} : \mathbb{P} \text{ is a probability measure and, }\\
        &\qquad \mean{|\eta|^\alpha} = \int_I |\eta|^\alpha \ \od \mathbb{P}(\eta) < \infty, \forall \alpha \le 2 + \delta \bigg\},
    \end{align*}
    for some $\delta > 0$. We can construct this density a random variable $Y$ with the properties, $\mean{Y} = \int_I Y \ \od \Theta(Y) = 0$, $\mean{Y^2} = \int_I Y^2 \ \od \Theta(Y) = 1$, then we choose $\eta_\nu$ with the property that $\mean{|\eta_\nu|^p}= \mean{|Y\sigma_\nu|^p}$ for $\nu = L, F$, and $p \in [0, 2 + \delta]$, with $0 < \sigma_\nu < 1$. Here $\sigma_\nu^2$ is the variance of the random variables $\eta_\nu$.

\section{Partial differential equations} \label{sec2:PDEs}
In this section, we discuss the system of Boltzmann-type equations
that results from the interactions \eqref{int:LL}--\eqref{int:FF},
and then discuss the system of Fokker-Planck-type equations, which are
obtained from the quasi-invariant limit of the Boltzmann-type system.
Throughout the rest of this paper, we make the simplifying assumption $u(w, v, t) = u(w, t)$.

        \subsection{Boltzmann-type system}\label{sec2:sub1:BoltFokker}
       
Using standard methods of kinetic theory we can now derive the
Boltzmann-type system associated with the binary interactions, for
completeness this process is outlined in Appendix~\ref{secA1}. Let $\phi$ be test functions from $C_c^{2,\delta}(\I;\real)$, that are continuous and compactly supported functions on $\I$.

The Boltzmann-type equation for the leader species $f_L(w,t)$ in
weak form is given by
\begin{equation} \label{Boltz:weaklead2}
    \begin{split}
        \nderiv{t}\left(\int_{\I} \phi(w) f_L(w,t) \ \od w\right) =& \frac{1}{\tau_{LL}}\Big \langle \int_{\I^2}\big[ \phi(w^*) + \phi(v^*) - \phi(w) - \phi(v) \big] \\
        &\qquad \qquad \times  f_L(w,t) f_L(v,t) \ \od w \od v \Big \rangle.% \\
    %    &+\beta \Big \langle \int_{\I^2} \big[ \phi(w^*) - \phi(w)\big] \\
     %   &\qquad \qquad\times f_L(w,t) f_F(v,t) \ \od w \od v \Big \rangle.
    \end{split}
\end{equation}

For the follower species $f_F(w,t)$ we have the Boltzmann-type equation in weak form,
\begin{equation} \label{Boltz:weakfollow2}
    \begin{split}
        \nderiv{t}\left(\int_{\I} \phi(w) f_F(w,t) \ \od w\right) =& \frac{1}{\tau_{FF}}\Big \langle \int_{\I^2}\big[ \phi(w^*) + \phi(v^*) - \phi(w) - \phi(v) \big] \\
        & \qquad \qquad \times f_F(w,t) f_F(v,t) \ \od w \od v \Big \rangle \\
        &+ \frac{1}{\tau_{FL}} \Big \langle \int_{\I^2} \big[ \phi(w^*) - \phi(w)\big] \\
        & \qquad \qquad \times f_F(w,t) f_L(v,t) \ \od w \od v \Big \rangle.
    \end{split}
\end{equation}
  In the above we denote by $\tau_{LL}$, $\tau_{FL}$ and $\tau_{FF}$
  the relaxation times associated with the respective interactions.

\subsection{Fokker-Planck-type system}

In the so-called quasi-invariant limit a coupled system of non-linear,
non-local Fokker-Planck equations is obtained from
\eqref{Boltz:weaklead2}, \eqref{Boltz:weakfollow2} for this model.
 To this end time is rescaled setting $s=\gamma_L t$ and the transformed
 density functions $g_L(w,s) = f_L(w,t)$ are considered, taking the limit $\sigma,
 \gamma_L \rightarrow 0$ while keeping the value $\sigma^2/\gamma_L$ constant, see Appendix~\ref{secA2} for further
 details. We then obtain the Fokker-Planck equation for the leader species,
    \begin{equation} \label{Fokker:lead}
        \begin{split}
            \pderiv{g_L}{s}(w, s) =& \npderiv{w} \left( \left( \frac{1}{\tau_{LL}} \mc{K}[g_L](w,s) + \frac{1}{2\tau_{LL}} u(w,s)\right) g_L(w,s) \right)\\
            &+ \frac{\lambda_{L} }{2\tau_{LL}}\nsecpderiv{w} (D^2(w) g_L(w,s)),
        \end{split}
      \end{equation}
      where
    \begin{equation*}
        \mc{K}[g_L](w,s) := \int_\I P_L(w,v)(w-v)g_L(v,s) \ \od v,
      \end{equation*}
      subject to no-flow boundary conditions.
      
 Taking the same limit and setting $\alpha_{LF} =\gamma_L/ \gamma_F$ (see Appendix~\ref{secA2} for further
 details), the follower Fokker-Planck-type equation is given by,
    \begin{equation} \label{Fokker:follow}
        \begin{split}
            \pderiv{g_F}{s}(w,s) =& \alpha_{LF} \npderiv{w} \left(\left(\frac{1}{2\tau_{FL}} \mc{M}[g_L](w,s)+ \frac{1}{\tau_{FF}}\mc{N}[g_F](w,s) \right)g_F (w,s)\right) \\
            &+  \left(\frac{\lambda_{F}}{4\tau_{LF}} + \frac{\lambda_{F}}{2\tau_{FF}}\right) \nsecpderiv{w} (D^2(w) g_F(w,s)),    
        \end{split}
    \end{equation}
    where
    \begin{align*}
         \mc{M}[g_L](w,s) :=& \int_\I \Tilde{P}(w,v)(w-v)g_L(v,s) \ \od v,\\
         \mc{N}[g_F](w,s) :=& \int_\I P_F(w,v)(w-v)g_F(v,s) \ \od v,
    \end{align*}
    subject to no-flow boundary conditions.
    
    The system \eqref{Fokker:lead} and \eqref{Fokker:follow} can be expressed as a matrix-vector equation, namely,
    \begin{equation} \label{Fokker:matvec}
        \pderiv{\mb{g}}{t} = \npderiv{w} \bigg( (A[\mb{g}] + B[u]) \mb{g} + \npderiv{w} (C\mb{g})\bigg),
    \end{equation}
    with,
    \begin{align*}
        \mb{g} &:= (g_L, g_F)^\top \text{, the ``state"}, \\
        A[\mb{g}] &:= \left( \begin{array}{cc}
            \frac{1}{\tau_{LL}}\mc{K}[g_L] & 0 \\
            0 & \alpha_{LF}\left(\frac{1}{2\tau_{FL}}\mc{M}[g_L] + \frac{1} {\tau_{FF}}\mc{N}_F[g_F] \right)
        \end{array} \right), \\
        B[u] &:= \left( \begin{array}{cc}
            \frac{u}{2\tau_{LL}} & 0 \\
            0 & 0 
        \end{array} \right), \\
    C &:= \left( \begin{array}{cc}
        \frac{\lambda_{L}}{2\tau_{LL}}D^2 & 0 \\
        0 & \left(\frac{\lambda_{F}}{4\tau_{FL}} + \frac{\lambda_{F}}{2\tau_{FF}}\right)D^2
    \end{array} \right), \\
    \mb{g}(0,w) &= \mb{g}^0(w).
    \end{align*}

    \section{PDE-constrained optimal control} \label{sec2:OCprobs}

    We introduce the optimal control problem with the system of
    Fokker-Planck-type equations \eqref{Fokker:matvec} as a PDE constraint. We derive the
    first-order optimality conditions using the formal Lagrangian Method. 
    
\subsection{Optimal control problem}
    
We now define the control problem for this model. The optimal control
of the model can come in three shapes: 1. microscopic -- which
optimises a cost functional with the microscopic interactions as the
state; 2. mesoscopic -- where the cost functional is minimised subject
to the system of Boltzmann-type equations \eqref{Boltz:weaklead2},
\eqref{Boltz:weakfollow2}; 3. macroscopic -- which minimises the cost
functional subject to the Fokker-Planck-type system. The hierarchy of
these optimal control problems for a single species are discussed in
\cite{bib:Albi:meancontrol}. Here, we are considering the optimal
control of the system of Fokker-Planck-type equations
arising for the two species leader-follower model. In the following,
we set out the macroscopic control problem (without pointwise control constraints).
    
\begin{Prob} \label{Fokker:Prob}
    Let $\mb{g}$ be the state functions and let $u$ be the associated control function, we wish to solve the following optimisation problem:
    \begin{equation}
        \begin{split}
            \min_{u\in {U}} J(\mb{g}, u),
        \end{split}
    \end{equation}
        subject to
    \begin{equation*}
        \pderiv{\mb{g}}{t} = \npderiv{w} \bigg( (A[\mb{g}] + B[u]) \mb{g} + \npderiv{w} (C\mb{g})\bigg),
    \end{equation*}
    with the initial conditions,
    \begin{equation*}
        \mb{g}(0,w) = \mb{g}^0(w),
      \end{equation*}
      and subject to no-flow boundary conditions.
    \end{Prob}
A typical form for the cost functional $J$ is a quadratic cost,
\begin{equation} \label{cost}
     J(\mb{g}, u) =     \frac{1}{2} \int_0^T \int_\I  \left(|w-w_{dF}|^2 + \beta |u|^2\right) g_F  \ \od w \od t,
   \end{equation}
where $U$ is the set of admissible controls and $\beta >0$. Minimisation of this cost
functional aims to concentrate opinions around a
given desired state $w_{dF}$ while also minimising the energy used by the control variable to
influence the state variable. Other choices of the cost functional are
investigated in Section~\ref{sec41:costfunctional}. In the following
we consider the set of admissible controls
$$
U=\big \{ ||u||_{L^2(0,T;L^\infty(\I))} \le M :u\in L^2(0,T;L^\infty(\I)) \big \},
$$
for some $M>0$. 

We can adapt results from \cite{bib:Albi:meancontrol} to establish the existence and
uniqueness of solutions to \eqref{Fokker:matvec} as well as the
existence of optimal controls for Problem~\ref{Fokker:Prob}. We denote
the dual space of $H^1(\I)$ by $H^{-1}(\I)$ and the duality pairing by
$\langle \cdot, \cdot \rangle_{{H^{-1},H^1}}: H^{-1}(\I)\times
H^{1}(\I)\to \mathbb{R},$ $\langle g, \varphi \rangle_{{H^{-1},H^1}}:=g(\varphi)$.

\begin{definition}\label{def:weaksol}
  For given $T>0$, functions $g_L, g_F:\I\times (0,T) \to [0,\infty)$
  are weak solutions of \eqref{Fokker:matvec} if and only if
  \begin{enumerate}
    \item $g_i \in L^2(0,T;H^1(\I))$, $\tfrac{\partial g_i}{\partial t} \in L^2(0,T;H^{-1}(\I))$, \quad
      $i=L,F,$
    \item For any $\varphi \in L^2(0,T;H^1(\I))$,
      \begin{align*}
             \int_0^T \Big\langle\pderiv{g_L}{t},\varphi
        \Big\rangle_{H^{-1},H^1}dt -& \int_0^T\int_{\I} \left( \left(
                                      \frac{1}{\tau_{LL}} \mc{K}[g_L]
                                      +\frac{1}{2\tau_{LL}} u\right)
                                      g_L \right) \frac{\partial
        \varphi}{\partial w}\\
            &+ \frac{\lambda_{L} }{2\tau_{LL}}\npderiv{w} (D^2(w) g_L) \frac{\partial
        \varphi}{\partial w} dw\,dt=0,\\
    \int_0^T \Big\langle\pderiv{g_F}{t} ,\varphi \Big\rangle_{H^{-1},H^1}dt -&\int_0^T\int_{\I} \alpha_{LF}\left(\left(\frac{1}{2\tau_{FL}} \mc{M}[g_L]+ \frac{1}{\tau_{FF}}\mc{N}[g_F](w,s) \right)g_F \right) \frac{\partial
        \varphi}{\partial w}\\
            &+  \left(\frac{\lambda_{F}}{4\tau_{LF}} +
              \frac{\lambda_{F}}{2\tau_{FF}}\right) \npderiv{w}
              (D^2(w) g_F(w,s)) \frac{\partial
        \varphi}{\partial w} dw\,dt=0.
      \end{align*}
    \end{enumerate}
  \end{definition}
\begin{theorem}\label{thm:existence}
  For given $T,M>0$, let $g_L^{0},g_F^{0}\in L^2(\I)$ and $u\in
  U$. Then there exist unique weak solutions $g_L,g_F$ to
  \eqref{Fokker:matvec} in the sense of Definition~\ref{def:weaksol}.
\end{theorem}%
\begin{proof}
The existence and uniqueness for the leader equation \eqref{Fokker:lead} follows from Theorem
2.4 in \cite{bib:Albi:meancontrol}. With $g_L$ given, we can follow along the lines of the
same proof for the follower equation \eqref{Fokker:follow}, the additional term in \eqref{Fokker:follow} arising from the
leader-follower interaction can be estimated similarly as the
other interaction term, noting $g_L\in L^2(0,T;H^1(\I)).$
\end{proof}

\begin{theorem}
   For given $T,M>0$, let $\mb{g}^0 =(g_L^{0},g_F^{0})\in L^2(\I)$. Then there
   exist an optimal control $\bar{u}\in U$ with associated optimal state $\mb{\bar{g}}=(\bar{g}_L, \bar{g}_F)$.
\end{theorem}
\begin{proof}
  The proof uses standard arguments, following similar steps as in Theorem
2.5 in \cite{bib:Albi:meancontrol}, presented here in concise form. Thanks to
Theorem~\ref{thm:existence}, for each $u\in U$ there exist unique
solutions $g_L,g_F$ to \eqref{Fokker:matvec}. Noting that the cost
functional is bounded from below by zero, we can consider 
minimising sequences $\mb{g}^{j} =(g_L^{j}, g_F^{j})$ and $u^{j}$ with
$$
\lim_{j\to \infty}   J(\mb{g}^{j}, u^{j})=\inf_{u\in U} J(\mb{g},u).
$$
We can employ Banach-Alaoglu's theorem to extract converging subsequences,
$$
g_L^{j_k},g_F^{j_k}\to
\bar{g}_L, \bar{g}_F\;\text{in}\;L^2(\I\times (0,T))\quad \text{and}\quad u^{j_k}\overset{\ast}{\rightharpoonup}  \bar{u}\;\text{in}\;L^2(0,T;L^\infty(\I))\quad\text{as}\;k\to\infty.
$$
Noting that $U$ is weak-$\ast$ closed, we have $\bar{u}\in U.$
The above regularity is sufficient to pass to the limit in the weak
formulation of \eqref{Fokker:matvec} and to conclude that
$\mb{\bar{ g}}=( \bar{g}_L, \bar{g}_F)$ solves \eqref{Fokker:matvec}. Semicontinuity of
the cost functional allows to conclude
$$
\inf_{u\in U} J(\mb{g},u)=\lim_{j\to \infty}   J(\mb{g}^{j},
u^{j})=\liminf_{k\to\infty}  J(\mb{g}^{j_k}, u^{j_k})\ge  J(\mb{\bar{ g}}, \bar{ u}),
$$
therefore $\bar{u}\in U$ is the optimal control with associated optimal state $\mb{\bar{g}}$.
\end{proof}

\subsection{First-order optimality conditions}\label{sec2:sub2:optcond}
    
Using the formal Lagrangian method for Problem~\ref{Fokker:Prob}, by introducing the adjoint variable $\mb{p} = (p_L, p_F)^\top$ as our Lagrange multipliers we construct the formal Lagrangian,
\begin{equation} \label{Fokker:Lagrange}
            \mc{L}(\mb{g},u,\mb{p}) = J(\mb{g},u) - \inner{\pderiv{\mb{g}}{s}}{\mb{p}} + \inner{\npderiv{w} ((A[\mb{g}]+B[u])\mb{g})}{\mb{p}} - \inner{\npderiv{w}(C\mb{g})}{\pderiv{\mb{p}}{w}},
\end{equation}
for which we must find the stationary point, $(\bar{\mb{g}}, \bar{u}, \bar{\mb{p}})$. To this end we compute the Fr\'echet derivatives of $\mc{L}$ with respect to $\mb{p}$, $\mb{g}$ and $u$, with the directions $\phi$, $\psi$ and $v$ respectively. That is we look for those functions $(\bar{\mb{g}}, \bar{u}, \bar{\mb{p}})$ such that,
\begin{align*}
    \nabla_\mb{p} \mc{L}(\bar{\mb{g}}, \bar{u}, \bar{\mb{p}})\phi &= 0, \\
    \nabla_\mb{g} \mc{L}(\bar{\mb{g}}, \bar{u}, \bar{\mb{p}}) \psi &= 0, \\
    \nabla_u\mc{L}(\bar{\mb{g}}, \bar{u}, \bar{\mb{p}})v &= 0.
\end{align*}
This leads to the weak form of the state system, $\nabla_\mb{p} \mc{L}\phi = 0$,
\begin{equation} \label{OCP:weakstate}
    \begin{split}
            \nabla_\mb{p}\mc{L}\phi = \inner{\pderiv{\mb{g}}{s}}{\phi} - \inner{\npderiv{w}((A[\mb{g}]+B[u])\mb{g})}{\phi} + \inner{\npderiv{w}(C\mb{g})}{\pderiv{\phi}{w}} =0,
    \end{split}
\end{equation}
the weak form of the adjoint system, $\nabla_\mb{g} \mc{L} \psi = 0$,
\begin{equation}\label{OCP:weakadjoint}
    \begin{split}
            \nabla_\mb{g}\mc{L}\psi =& \nabla_\mb{g} J \psi + \inner{\pderiv{\mb{p}}{s}}{\psi} + \inner{ \pderiv{\mb{p}}{w} }{ (A[\mb{g}] + B[u])\psi} + \inner{ \pderiv{\mb{p}}{w}}{A[\psi]\mb{g}} \\
    &- \inner{\pderiv{\mb{p}}{w}}{\npderiv{w} (C \psi)} - \int_{\I} \mb{p}(T) \cdot \psi(T) \ \od w =0,
    \end{split}
\end{equation}
and finally the optimality condition, $\nabla_u\mc{L}v = 0$,
\begin{equation}\label{OCP:weakopt}
    \begin{split}
            \nabla_u\mc{L}v =& \nabla_u J v + \inner{\npderiv{w}(B[v]\mb{g})}{\mb{p}} = 0.
    \end{split}
\end{equation}
where we omit the bars for clarity.
We can now write \eqref{OCP:weakstate}--\eqref{OCP:weakopt} in their corresponding strong forms. These strong forms are somewhat unwieldy and therefore, we make a small assumption on the cost-functional $J$, to present the above system in more readable way -- that is we assume that the cost functional will have the general form $J(\bg, \bu) = J_T[\bg] + J_s[\bg, \bu]$, that is, $J$ is separated into a time dependent part $J_s$ and a final time part $J_T$. For example the typical choice for $J$ stated in Section~\ref{sec2:sub1:BoltFokker}, the decomposition is $J_T[\bg] = 1/2 \norm{\bg(T) - \bg_\I}_{L^2(\I;\real^2)}^2$, $J_s[\bg, \bu] = 1/2 \norm{\bu}_U^2$. We then write these equations in the strong form beginning with the state system,
\begin{equation} \label{OCP:strongstate}
    \begin{split}
        e(\mb{g}, u) := \pderiv{\mb{g}}{s} = \npderiv{w}((A[\mb{g}]+B[u])\mb{g}) + \nsecpderiv{w}(C\mb{g}),
    \end{split}
\end{equation}
with initial condition $\bg(0, w) = \mathbf{h}_0(w)$. We then write the strong form of the adjoint system,
\begin{equation} \label{OCP:strongadjoint}
 e^*(\mb{g}, u, \mb{p}) := \nabla_\mb{g} J_s + \pderiv{\mb{p}}{s} - (A[\mb{g}] + B[u])\pderiv{\mb{p}}{w} + A^*\left[\mb{g},\pderiv{\mb{p}}{w}\right] + C\secpderiv{\mb{p}}{w} =0,
\end{equation}
with the final time condition $\bp(T, w) = \del_{\bg} J_T[\bg]$ and finally the strong form of the optimality condition,
\begin{equation} \label{OCP:strongoptcond}
    \varepsilon(\mb{g}, u, \mb{p}) = \nabla_u J_s + B^*[\bg]\pderiv{\mb{p}}{w} = 0.
\end{equation}
with the associated boundary conditions. In these equations, we required writing some extra terms, specifically,
\begin{align*}
    A^*\left[\bg, \pderiv{\bp}{w}\right] &:= \left( \begin{array}{cc}
            \frac{1}{\tau_{LL}}    \kappa[g_L, \pderiv{p_L}{w}] + \frac{\alpha_{LF}}{2 \tau_{FL}} \mu[g_F, \pderiv{p_F}{w}] \\
            \frac{\alpha_{LF}}{\tau_{FF}} \nu[g_F, \pderiv{p_F}{w}]
        \end{array} \right), \\
    B^*[\bg] &:= \left( \begin{array}{cc}
            \frac{1}{2\tau_{LL}}g_L & 0 \\
            0 & 0 
        \end{array} \right), \\
    \kappa\left[g_L, \pderiv{p_L}{w}\right] & := \int_{\I} P_L(w,v)(v-w) g_L  \pderiv{p_L}{w} \ \od w ,\\
    \mu\left[g_F, \pderiv{p_F}{w}\right] &:= \int_{\I} \Tilde{P}(w,v)(v-w) g_F  \pderiv{p_F}{w} \ \od w ,\\
        \nu\left[g_F, \pderiv{p_F}{w}\right] & := \int_{\I} P_F(w,v)(v-w) g_F  \pderiv{p_F}{w} \ \od w.
\end{align*}
These terms were computed using a similar method to that in \cite{bib:Albi:meancontrol}. 

 \section{Numerical methods} \label{sec3:Nummeth}
      In order to solve the optimality system from Section~\ref{sec2:OCprobs} numerically, we use a gradient-type sweeping algorithm \cite{bib:Burgeretal:OCPMFG}. This involves making an initial guess for the control function, $u_{-1}$, and then creating a sequence of functions $(\mb{g}_i, u_i, \mb{p}_i)$ using the following algorithm:
\begin{Alg}\label{alg1}
With the initial guess $u_{-1}$ for the control function, set $n = 0$,
    \begin{enumerate}[(1)]
        \item Solve the state system $e(\mb{g_n}, u_{n-1}) = 0$ forwards in time. \label{number1}
        \item With the known $\mb{g}_n$, solve the adjoint system $e^*(\mb{g}_n, u_{n-1}, \mb{p}_n)$ backwards in time.
        \item Update the control function using $u_n = u_{n-1} - \nu \varepsilon(\mb{g}_n, u_{n-1}, \mb{p}_n)$, for some chosen step size $\nu >0$.
        \item  Set $n = n+1$, return to step \ref{number1}. until convergence of the cost functional $J$.
    \end{enumerate}
\end{Alg}
When solving the state system numerically some additional challenges present themselves. Since the solution of the state equation is a probability density function we require that $\bg_n$ is unconditionally positive and the mass is conserved -- $\int_\I g_{L, n}(t, w) \ \od w = \int_\I g_{F, n}(t, w) \ \od w = 1$ -- for all $s$. Additionally since the compromise term $A[\mb{g}_n] \bg_n$ is non-linear and non-local special care must be taken when choosing a numerical method.

To combat these challenges, we use a Chang-Cooper method to discretise
in space -- which is second-order consistent in space and preserves
quasi-steady states -- and the modified Patankar-Runge-Kutta (MPRK)
scheme to discretise in time -- which is unconditionally positive \cite{bib:Bartel}. One
main advantage of this method over an explicit Euler method is the lack of constraints on the time step for convergence.

We partition our space axis with $L$ equidistant nodal points over $\I$, that is $w_i = -1 + i \Delta w$ with $\Delta w = 2/L$, and we partition our time axis with an equidistant mesh over $[0,T)$ with some $\Delta s$ chosen and $Q = \lfloor T/\Delta s \rfloor$ nodal points. We write a lower index to denote the spatial node value and an upper index to denote the temporal node value, that is $\bg_n(w_i,s^t) = \bg_{n, i}^t$. At each iteration in Algorithm 1 the values of $n$ are fixed, we omit these for ease of reading.

We now discuss the specifics of the discretisation in more detail. Since we are using a Chang-Cooper scheme for discretisation in space, we consider the semi-discretisation for \eqref{OCP:strongstate},
\begin{equation} \label{changcoop}
  \deriv{\mb{g}_i}{t} = \frac{\F_{\iplus}-\F_{\iminus}}{\Delta w},
\end{equation}
for the numerical fluxes,
\begin{equation*}
    \F_{i+\half} = \alpha_{\iplus} \left( \left(1-\delta_{\iplus} \right) \bg_{i+1} + \delta_{\iplus} \bg_i \right) + C_{\iplus}\frac{\bg_{i+1}-\bg_i}{\Delta w},
\end{equation*}
with
\begin{align*}
  \alpha_{\iplus} & = A[\bg] \left( w_{\iplus} \right) + B[u] \left( w_{\iplus} \right) + C^\prime_{\iplus}, \\
  \delta_{\iplus} & = \frac{1}{1-\exp \left( \lambda_{\iplus} \right) } + \frac{1}{\lambda_{\iplus}}, \\
  \lambda_{\iplus} & = \frac{\alpha_{\iplus} \Delta w }{C_{\iplus}}, \\
  C_{\iplus} & = C\left(w_{\iplus}\right).
\end{align*}
With the semi-discretisation in space, we can write this equation in a production-destruction formulation, that is,
\begin{equation}
    \deriv{\mb{g}_i}{t} = \sum_{j= 0} ^ L p_{i, j}(\bg) - d_{i, j}(\bg), \qquad \forall i = 0, 1,..., L
\end{equation}
with the corresponding initial condition and $\bg$ representing the vector of time dependent functions and $p_{i,j}$ and $d_{i,j}$ representing the contribution to nodal point $i$ from nodal point $j$ -- the production -- and the loss from nodal point $j$ to nodal point $i$ -- the destruction. Note that $p_{i,j} = d_{j,i}$. Since the solution of our state system are probability density functions we require that our scheme be conservative which allows us to narrow down how we can choose the functional $p$ and $d$, namely we wish for $p_{i,i} = d_{i,i} = 0$ since this leads to
\begin{equation*}
    \sum_{i=0}^L \deriv{\mb{g}_i}{t} = \sum_{i=0}^L\sum_{j= 0} ^ L p_{i, j}(\bg) - d_{i, j}(\bg) = \sum_{i=0}^L p_{i,i} - d_{i,i} = 0,
\end{equation*}
which is equivalent to the conservation of mass for the semi-discretised functions $\bg_i$.

We use the proposed values for $p$ and $d$ presented in \cite{bib:Bartel} -- where a more in depth description of this method can be found -- namely,
\begin{align*}
  p_{i,i+1}[\bg] & = d_{i+1,i} := \frac{\max \{0,\alpha_{\iplus} \} \left( \left( 1 - \delta_{\iplus} \right) \bg_{i+1} + \delta_{\iplus} \bg_i \right) + \frac{C_{\iplus} \bg_{i+1}}{\Delta w} }{\Delta w}, \\
  p_{i,i-1} [\bg] & = d_{i-1,i} := \frac{-\min \{ 0,\alpha_{\iminus} \} \left( \left(1-\delta_{\iminus} \right) \bg_i + \delta_{\iminus} \bg_{i-1} \right) + \frac{C_{\iminus} \bg_{i-1}}{\Delta w} }{\Delta w}, \\
  p_{i,j} & = d_{i,j \phantom{-1}} := 0 \qquad \forall \ j \notin \{ i+1,i-1 \}.
\end{align*}
Finally, we can use the Patankar trick along with a second-order Runge-Kutta scheme for the time integration of \eqref{changcoop},
\begin{align*}
  \overline{\bg}_i & = \bg^t_i + \Delta s \left( \sum_{j =0}^L p_{i,j} [\bg^t] \frac{\overline{\bg}_j}{\bg^t_j} - \sum_{j =0}^L d_{i,j} [\bg^t] \frac{\overline{\bg}_i}{\bg^t_i} \right), \\
  \bg^{t+1}_i & = \bg^t_i + \frac{\Delta s}{2} \left( \sum_{j =0}^L \left( p_{i,j} [\bg^t] + p_{i,j} [\overline{\bg} ] \right) \frac{\bg^{t+1}_j}{\overline{\bg}_j} - \sum_{j =0}^L \left( d_{i,j} [\bg^t ] + d_{i,j} [\overline{\bg}] \right) \frac{\bg^{t+1}_i}{\overline{\bg}_i} \right).
\end{align*}

For the adjoint problem \eqref{OCP:strongadjoint} the terms which
require additional care are the non-local terms, $A[\bg] \od\bp/\od w$
and $A^*[\bg, \od\bp/\od w]$. We use an upwind scheme to discretise
the adjoint problem in space and discretise in time using a modified
forward Euler scheme. The slight modification to the forward Euler
scheme comes from treating the $A^*$ integral term in \eqref{OCP:strongadjoint} explicitly with respect to $\bp$. That is,
\begin{equation*}
  - \left( \bp^t - \bp^{t+1} \right) - \Delta t \left( A [\bg^t] + B [u^t] \right) \pderiv{\bp^t}{w} + \Delta t C_i \secpderiv{\bp^t}{w} + A^* \left[ \bg^t, \pderiv{\bp^{t+1}}{w} \right] + \nabla_{\bg} J_s \left( \bg^t,u^t \right) = 0 ,
\end{equation*}
which, after isolating the $\bp^t$ terms on the left hand side and the other terms on the right hand side, yields,
\begin{align*}
  \left( \frac{\Delta s}{\Delta w} \kappa_{-}  - \frac{\Delta s}{\Delta w^2} C_i \right) & \bp^t_{i+1} + \left( 1 + \frac{2\Delta s}{\Delta w^2} C_i + \frac{\Delta s}{\Delta w}\left( \kappa_{+} - \kappa_{-} \right) \right) \bp^t_i + \left( - \frac{\Delta s}{\Delta w}\kappa_{+} - \frac{\Delta s}{\Delta w^2} C_i \right) \bp^t_{i-1} \\
  & = \bp^{t+1}_i + \Delta s \nabla_{\bg}J_s \left( \bg^t_i,u^t_i \right) - A^* \left[ \bg^t, \pderiv{\bp^{t+1}}{w} \right] ,
\end{align*}
with
\begin{align*}
  \kappa_{+} & = \max \{ 0, A[\bg^t](w_i) + B [u^t](w_i) \}, \\ 
  \kappa_{-} & = \min \{ 0, A[\bg^t](w_i) + B [u^t](w_i) \}. 
\end{align*}
Finally, the update of the control function is as described in Algorithm 1, evaluated at all space and time grid points.

 \section{Numerical experiments} \label{sec4:Numexp}

In this section, we show the results of some numerical experiments, exploring the behaviour of these models using different compromise functions $P$ and cost functionals $J$. We take the initial guess of the control function in all experiments as $u \equiv 0$ for all $(w, t) \in \I \times [0, T]$, and define the state at time $t=0$ as localisation functions described in terms of its leader and follower components as,
\begin{align}
    h_{L,0}(w) & = \Lambda_L (\tanh(k(R_L - |w - c_L|)) + 1), \\
    h_{F,0}(w) & = \Lambda_F (\tanh(k(R_F - |w - c_F|)) + 1),
\end{align}
where $R_L, R_F$ are the radii of the graph and $c_L, c_F$ are the centres. In this equation, $\Lambda_L$ and $\Lambda_F$ are the constants which ensure, $$\int_\I h_{L,0} (w) \ \od w = \int_\I h_{F,0} (w) \ \od w = 1,$$ and $k$ is the sharpness factor of the hyperbolic tangent function. In particular we have chosen $R_L = R_F = 0.85$, $c_L = c_F = 0$, $k = 10$.

    \begin{figure}[h]
	    \begin{subfigure}[b]{0.45\linewidth} 
	        \centering
		    \includegraphics[scale=0.42]{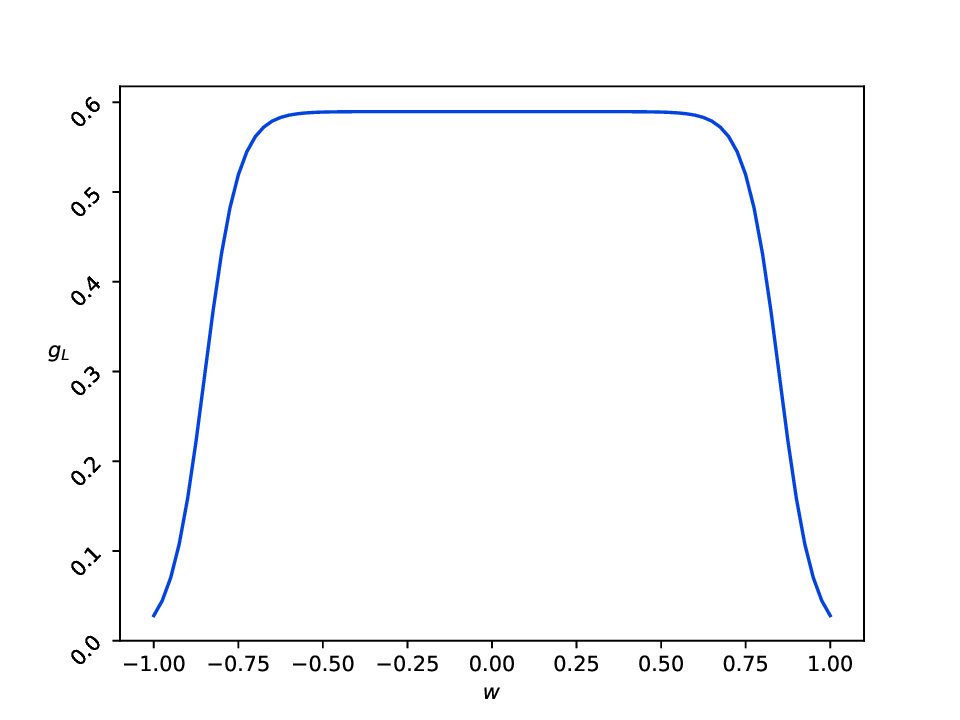}
    		\caption{Initial time leader state $g_L^0$}
    		% \label{fig:4.3analysis:Not:2015act}
	    \end{subfigure}
    \hspace{0.04cm}
    	\begin{subfigure}[b]{0.45\linewidth} 
    	\centering
	    	\includegraphics[scale=0.42]{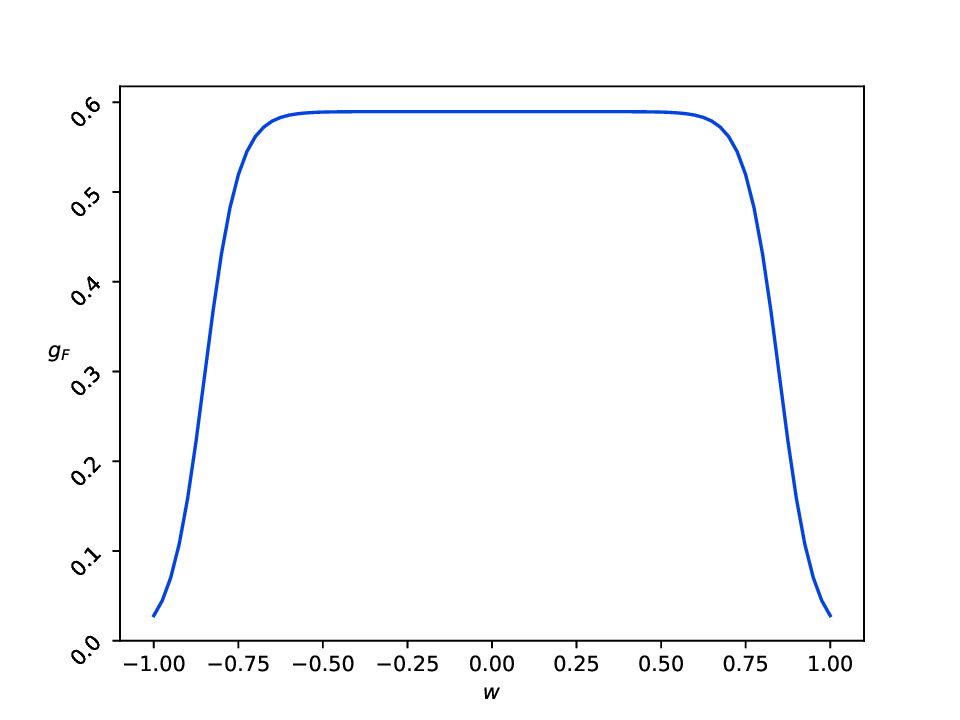}
		    \caption{Initial time follower state $g_F^0$}
		    % \label{fig:4.3analysis:Not:2019act}
	    \end{subfigure}
    	\caption{Densities for the leader and follower species at time $t=0$.}
    \end{figure}

For all the following experiments, we partition our opinion axis with
$L=80$ equidistant nodal points in $\I$, that is $\Delta w = 2/80 =
0.025$. Due to the absence of any time-stepping restriction in the
MPRK method, we are free to choose a relatively large $\Delta s$ and we make the choice of $\Delta s = 5 \Delta w$ to reduce the time to run the simulations.

In Table~\ref{parameters}, we report the parameters used in the numerical experiments for the Fokker-Planck modelling.

\begin{table}[h]
    \caption{Default parameters used for the Fokker-Planck simulations.}
        \label{parameters}
    \centering
    \begin{tabular}{c|cccccccc}
         Parameter & $\tau_{LL}$ & $\tau_{FL}$ & $\tau_{FF}$ &
                                                               $\lambda_L$ & $\lambda_F$ & $r$ & $\beta$\\
         \hline
          Value & 0.2 & 2 & 0.2 & 0.05 & 0.05 & 0.5& 0.05
    \end{tabular}
\end{table}

\subsection{Investigating the choice of cost functional $J$}
\label{sec41:costfunctional}
     
In this numerical experiment, we explore the behaviour of our model using some example cost functionals $J$ and compare these to the uncontrolled dynamics -- a top view of the time evolution of which are presented in Figure~\ref{heatmapuncont}. For this we use a bounded confidence model for our compromise functions, $P_L = P_F = \Tilde{P} = P$ as follows,
$$P(x, y):= \begin{dcases}
    1 & \text{for } |x-y| < r, \\
    0 & \text{otherwise},
\end{dcases}$$ for $r = 0.5$.

    \begin{figure}[htbp]
	    \begin{subfigure}[b]{0.45\linewidth} 
	        \centering
		    \includegraphics[scale=0.44]{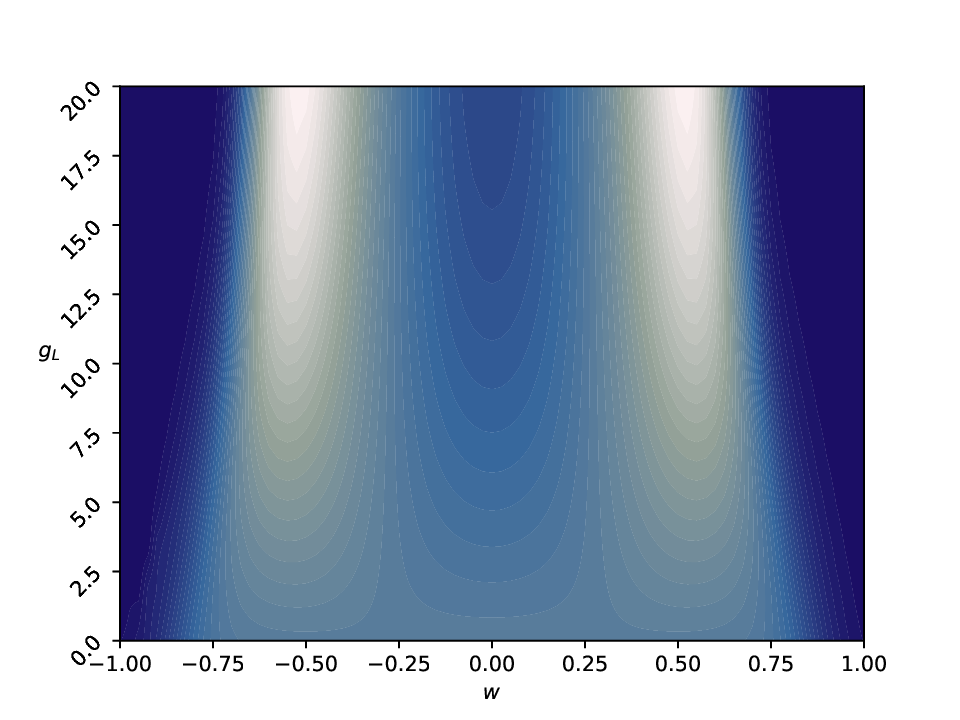}
    		\caption{Leader time evolution $g_L$}
    		\label{fig:uncontleader}
	    \end{subfigure}
    \hspace{0.04cm}
    	\begin{subfigure}[b]{0.45\linewidth} 
    	\centering
	    	\includegraphics[scale=0.44]{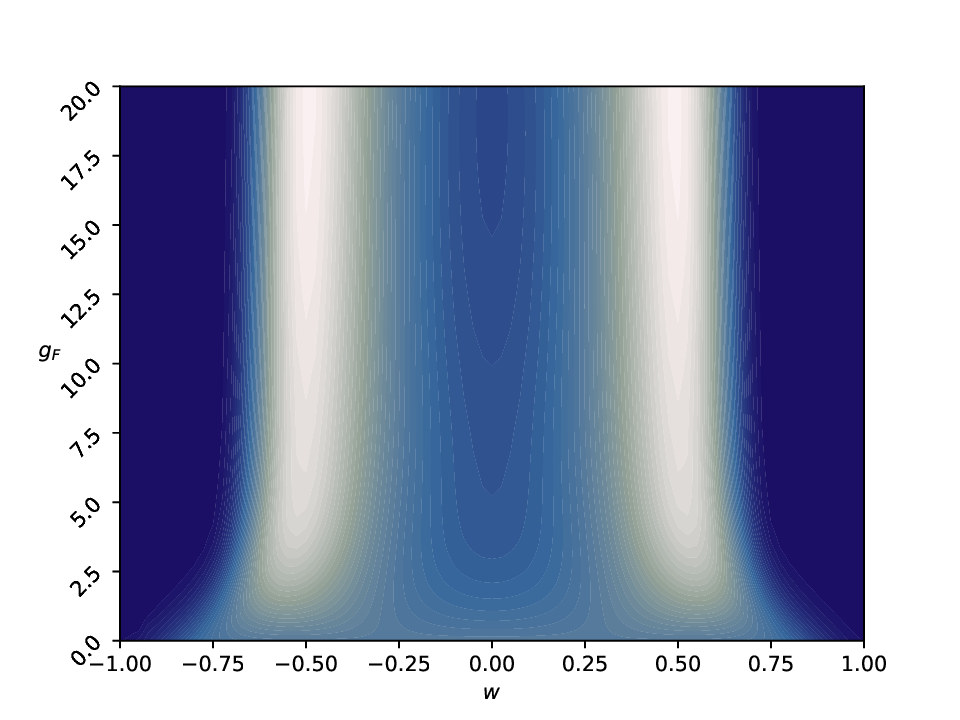}
		    \caption{Follower time evolution $g_F$.}
		    \label{fig:uncontfollower}
	    \end{subfigure}
    	\caption{Top-down view of the time evolution of the
          uncontrolled densities $g_L$, $g_F$.}
     \label{heatmapuncont}
    \end{figure}

The first example cost functional is similar to that from \cite{bib:Albi:meancontrol, bib:AlbiPareschiZanella:MPC}, which mirrors a common type of cost functional for the binary controlled dynamics \eqref{int:LL}--\eqref{int:FF},
\begin{equation} \label{centring}
    J(\bg, u) = \int_0^T \int_\I \left( \frac{1}{2} \left(   |w-w_{dL}|^2 g_L + \beta |u|^2 g_L \right) + \frac{1}{2} \left(   |w-w_{dF}|^2 g_F + \beta |u|^2 g_F \right) \right) \ \od w \od t,
\end{equation}
where $w_{dL}, w_{dF}$ are some desired opinions that we wish the densities $g_L$ and $g_F$ to be centred around respectively.

    \begin{figure}[htbp]
	    \begin{subfigure}[b]{0.5\linewidth}
		\centering
		    \includegraphics[scale=0.4]{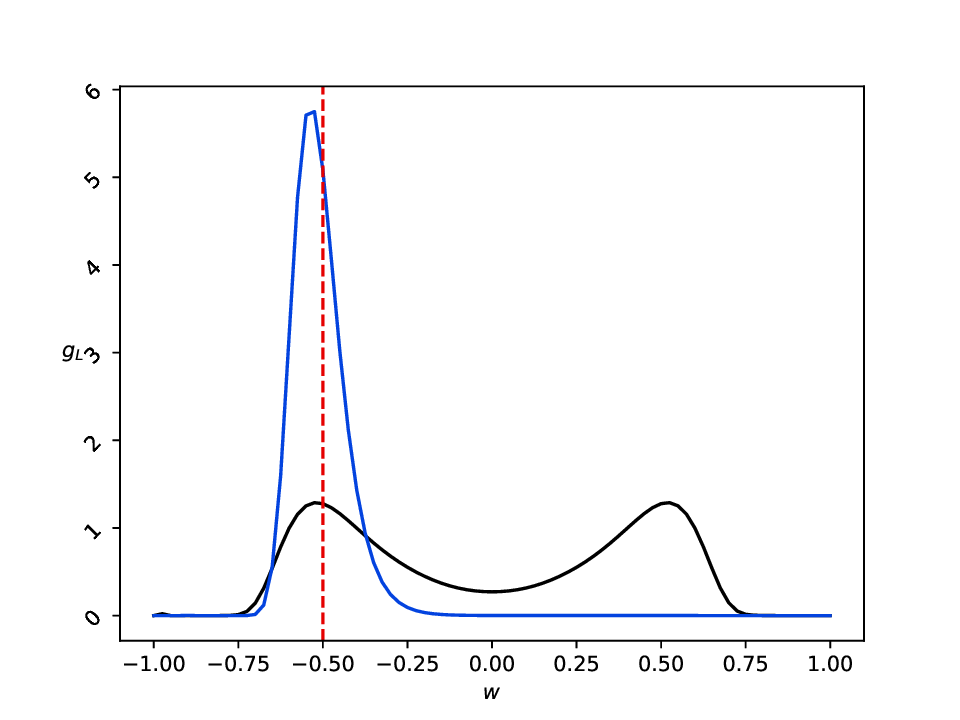}
		    \caption{Final time leader density.}
	    \end{subfigure}
		\hspace{0.04cm}
	    \begin{subfigure}[b]{0.5\linewidth}
		    \centering
		    \includegraphics[scale=0.4]{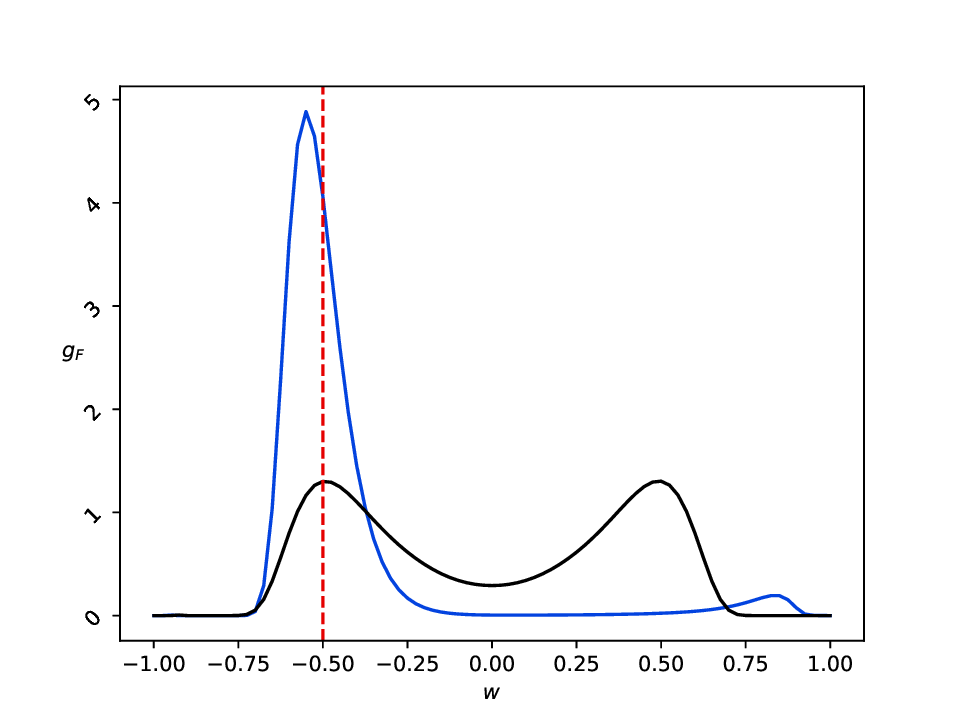}
		    \caption{Final time follower density.}
	    \end{subfigure} 
    
    	\begin{subfigure}[b]{0.5\linewidth}
    	\centering
	    	\includegraphics[scale=0.4]{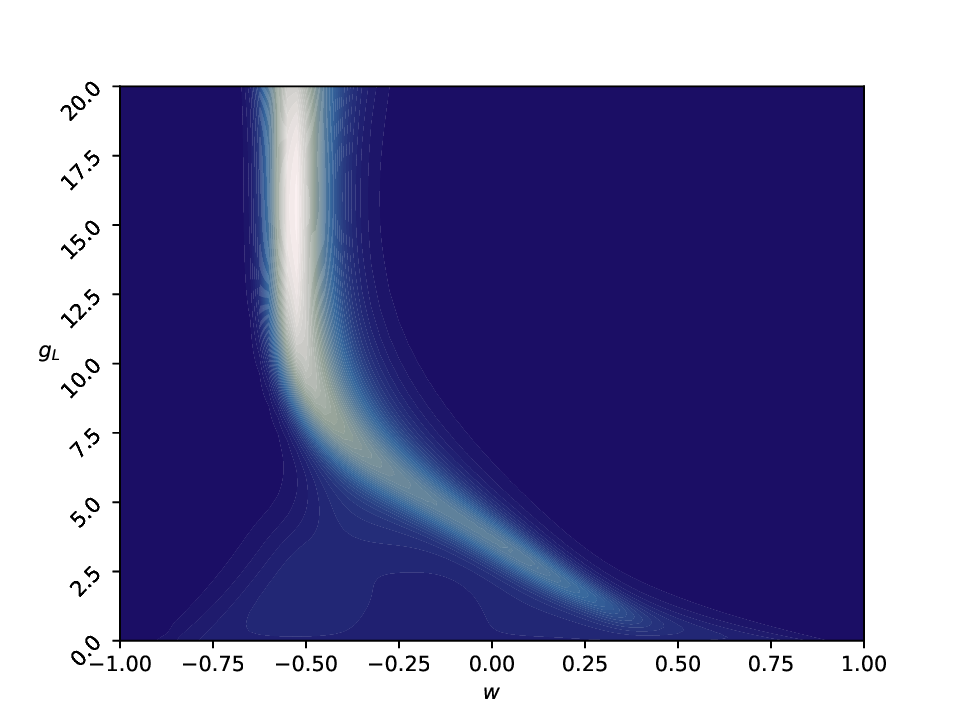}
		    \caption{Leader time evolution $g_L$.}
	    \end{subfigure} 
	    \hspace{0.04cm}
	    \begin{subfigure}[b]{0.5\linewidth}
		    \centering
		    \includegraphics[scale=0.4]{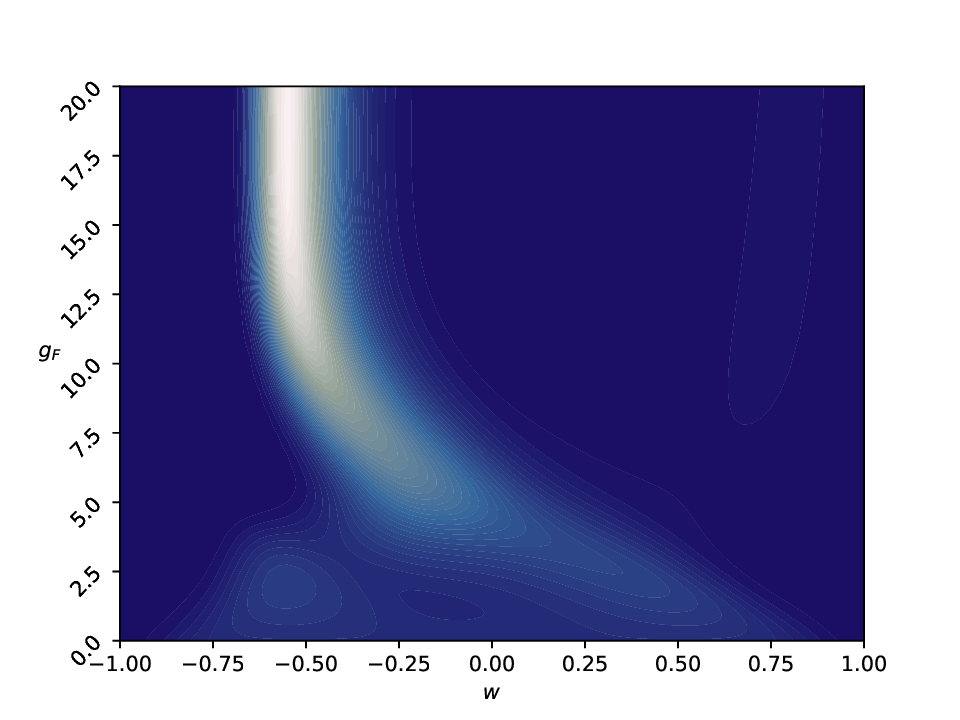}
            \caption{Follower time evolution $g_F$.}
	    \end{subfigure}
        
	    \hspace{0.25\linewidth}
        	    \begin{subfigure}[b]{0.5\linewidth}
		\centering
		    \includegraphics[scale=0.4]{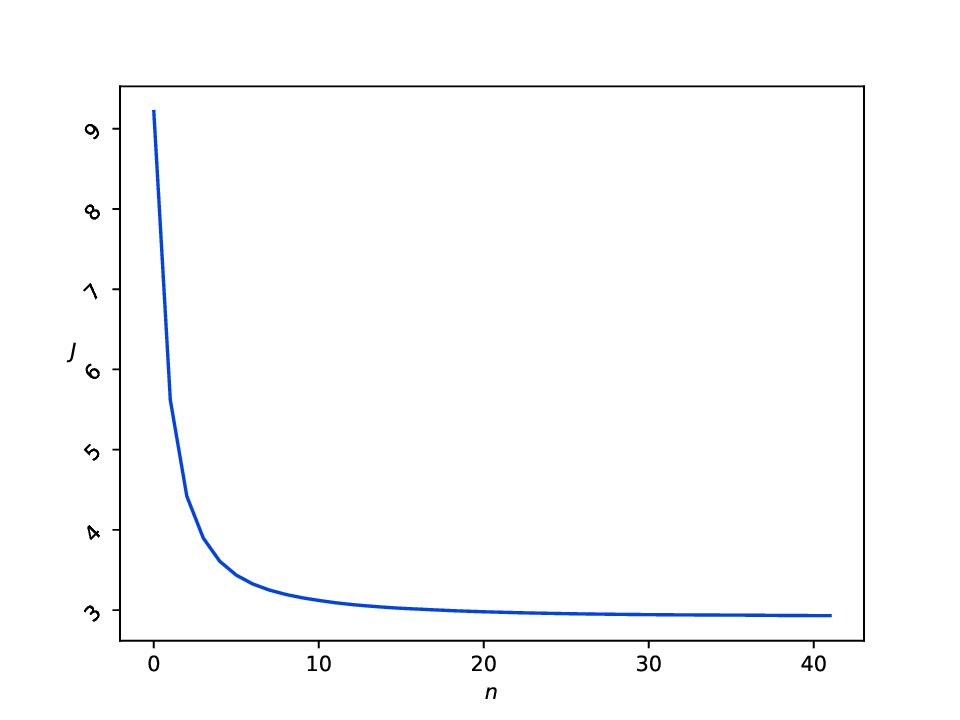}
		    \caption{Cost functional.}
            % \label{centringres}
	    \end{subfigure}
	    \caption{The top set of graphs represent the final time states for the leader and follower species (solid blue), the uncontrolled dynamics (solid black) and the desired opinion $w_d$ according to the cost functional \eqref{centring} (dashed red). The middle set of graphs represent the top-down view of the time evolution of the leader and follower species. The bottom graph shows the values of the cost functional \eqref{centring} evaluated at each step of Algorithm~\ref{alg1}.}
	    \label{centringgraphs}
    \end{figure}

We present the resulting states for the cost functional \eqref{centring} in Figure~\ref{centringgraphs} using $w_{dL}= w_{dF} = w_d = -0.5$. As we can see from these figures, the controlled states, mainly cluster around the desired position $w_d$ in this case, whereas the uncontrolled states form two distinct and equal clusters. In the follower dynamics we see a peak forming, centred at $w=0.8$, which -- when looking at the time evolutions of the state -- appears to be a concentration of followers that are outside the confidence radius $r=0.5$ of the leader species in the early time steps. This might be due to the fast convergence of the leaders to the desired state $w_d$, thus ``leaving'' some followers behind.

To explore this slightly more, we examine a variation of cost functional \eqref{centring}, where the cost functional depends only on the follower density rather than depending on both the leader and follower densities,

\begin{equation}\label{Fcentring}
    J(\bg, u) =  \frac{1}{2} \int_0^T \int_\I  \left(|w-w_{dF}|^2 + \beta |u|^2\right) g_F  \ \od w \od t.
\end{equation}

    \begin{figure}
	    \begin{subfigure}[b]{0.5\linewidth}
		\centering
		    \includegraphics[scale=0.4]{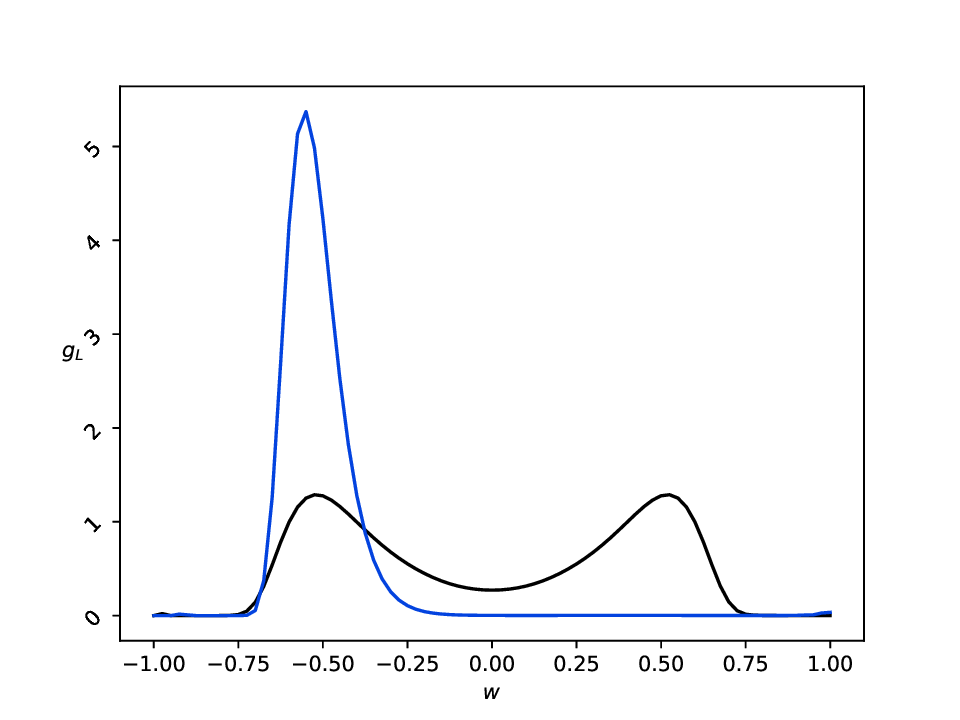}
		    \caption{Final time leader density.}
	    \end{subfigure}
		\hspace{0.04cm}
	    \begin{subfigure}[b]{0.5\linewidth}
		    \centering
		    \includegraphics[scale=0.4]{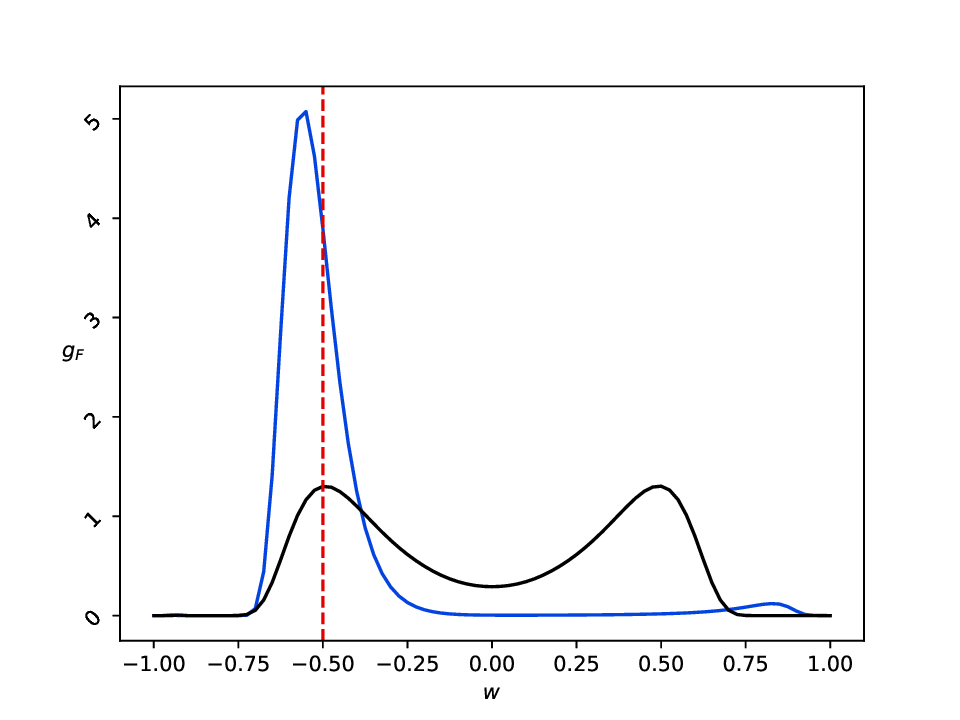}
		    \caption{Final time follower density.}
	    \end{subfigure} 
    
    	\begin{subfigure}[b]{0.5\linewidth}
    	\centering
	    	\includegraphics[scale=0.4]{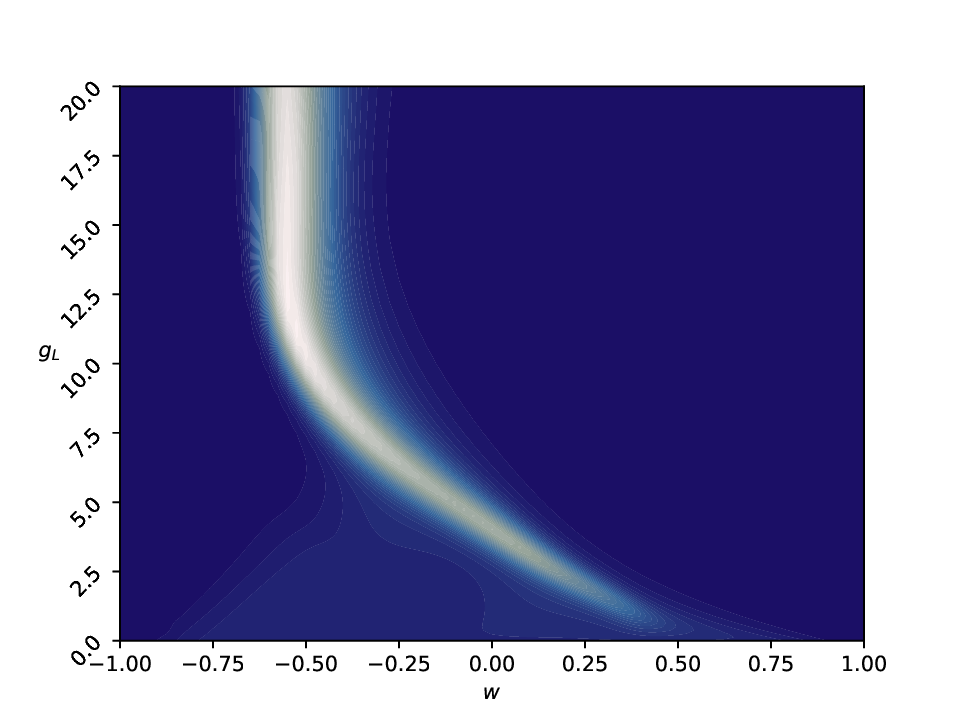}
		    \caption{Leader time evolution $g_L$.}
	    \end{subfigure} 
	    \hspace{0.04cm}
	    \begin{subfigure}[b]{0.5\linewidth}
		    \centering
		    \includegraphics[scale=0.4]{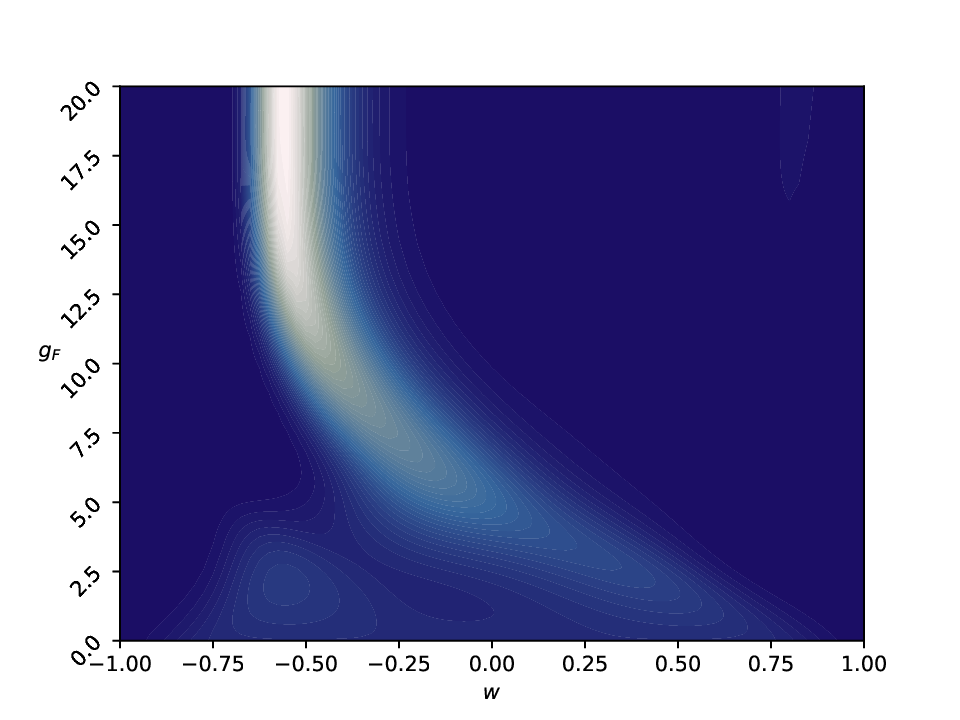}
            \caption{Follower time evolution $g_F$.}
	    \end{subfigure}

        	    \hspace{0.25\linewidth}
        	    \begin{subfigure}[b]{0.5\linewidth}
		\centering
		    \includegraphics[scale=0.4]{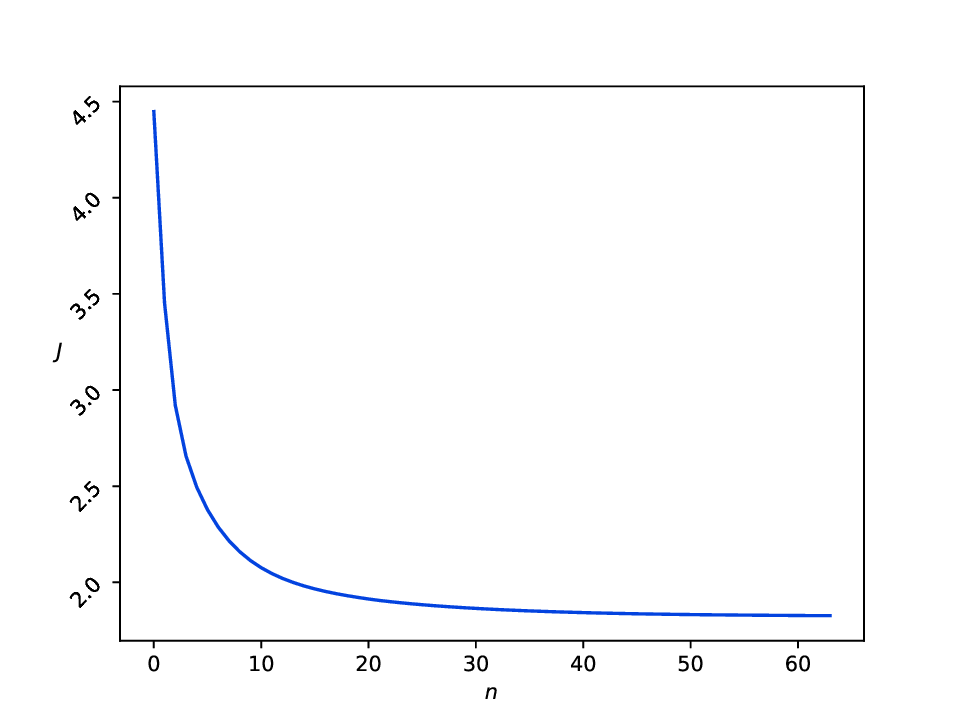}
		    \caption{Cost functional.}
            % \label{Fcentringres}
	    \end{subfigure}
	    \caption{The top set of graphs represent the final time states for the leader and follower species (solid blue), the uncontrolled dynamics (solid black) and the desired opinion $w_d$ according to the cost functional \eqref{Fcentring} (dashed red). The middle set of graphs represent the top-down view of the time evolution of the leader and follower species. The bottom graph is the values of the cost functional \eqref{Fcentring} evaluated at each step of Algorithm~\ref{alg1}.}
	    \label{Fcentringgraphs}
    \end{figure}

From the results presented in Figure~\ref{Fcentringgraphs}, we see the leader species follows a very similar strategy as in the previous experiment, however the follower species still forms a peak at around $w=0.8$, although this peak has less mass. A more illustrative example of this change in strategy for the leader species is where we modify the nature of the interactions between the leaders and followers. In this example we take $$\Tilde{P}(x, y) = b(1-x^2),$$ known as a Sznajd-type interaction \cite{sznajd2000opinion} for some constant $b$ and where $\Tilde{P}$ is the compromise function that appears in the follower side of the leader-follower interaction \eqref{int:FL}.

Since this compromise function only depends on the interacting individuals' opinions, this type of interaction can be seen as the propensity for individuals to change their opinion, where individuals with extreme opinions -- close to $-1$ and $1$ -- are influenced little by interactions and individuals with central opinions -- close to $0$ -- are influenced greatly by the interactions. The constant $b$ is used to describe the weight of the interaction and in general we allow $b \in [-1,1]$ to ensure that $|\Tilde{P}|<1$, with $b >0$ representing a concentration of the follower density around the leader density and $b<0$ representing the separation of the followers from the leaders.

In this numerical experiment we use $b=-1$, to ensure that either the leader species or the follower species can converge on the desired opinion but not both. Recall that the interactions within the leader and follower species are still using the bounded confidence-type interactions -- interactions only occur when the opinions of individuals are close together. This leads to radically different strategies for the leader species depending on the choice of cost functional.

    \begin{figure}
	    \begin{subfigure}[b]{0.5\linewidth}
		\centering
		    \includegraphics[scale=0.4]{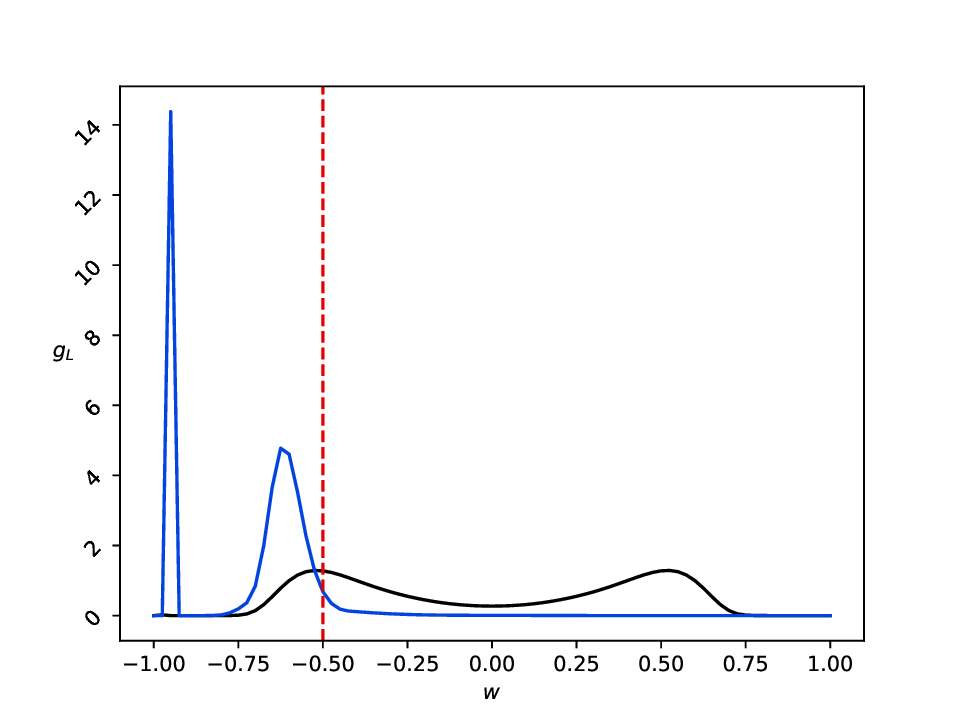}
		    \caption{Final time leader density.}
	    \end{subfigure}
		\hspace{0.04cm}
	    \begin{subfigure}[b]{0.5\linewidth}
		    \centering
		    \includegraphics[scale=0.4]{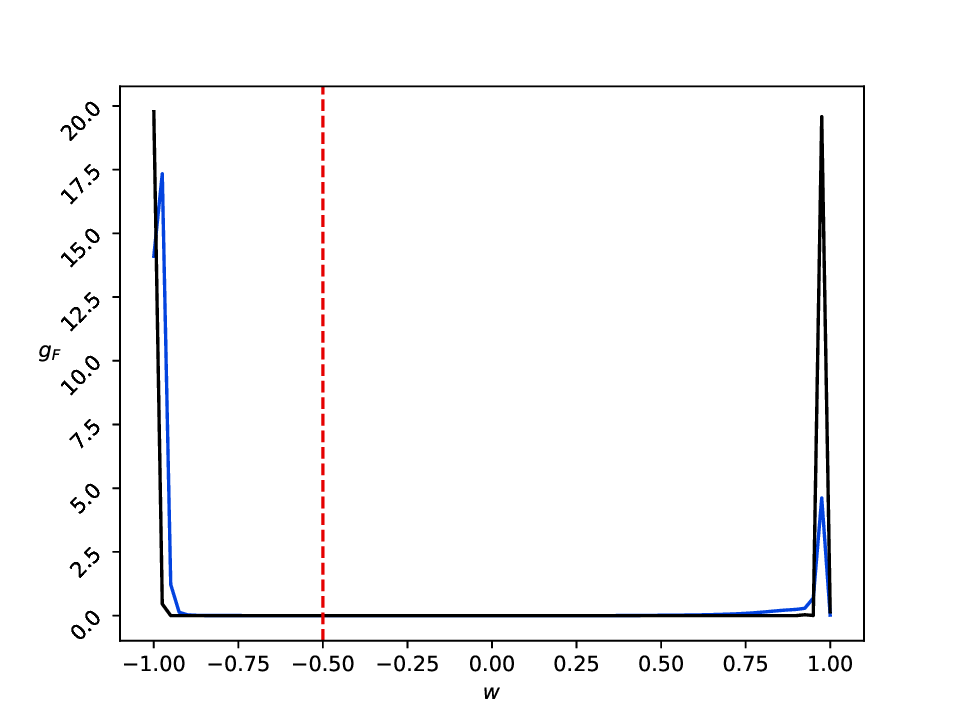}
		    \caption{Final time follower density.}
	    \end{subfigure} 
    
    	\begin{subfigure}[b]{0.5\linewidth}
    	\centering
	    	\includegraphics[scale=0.4]{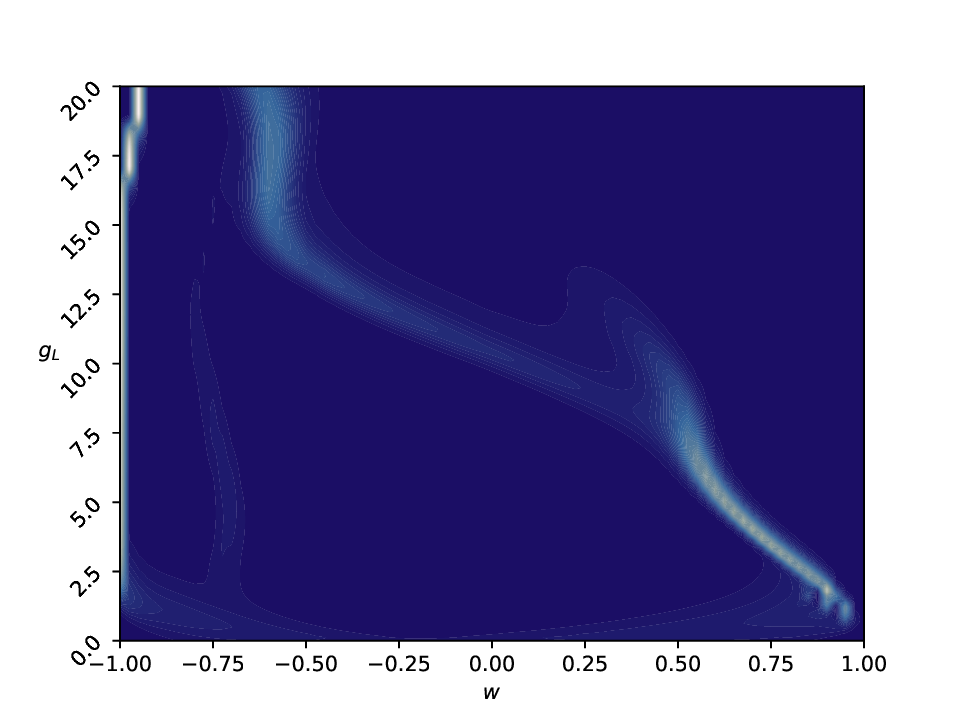}
		    \caption{Leader time evolution $g_L$.}
                        \label{sznleaderdynam}
	    \end{subfigure} 
	    \hspace{0.04cm}
	    \begin{subfigure}[b]{0.5\linewidth}
		    \centering
		    \includegraphics[scale=0.4]{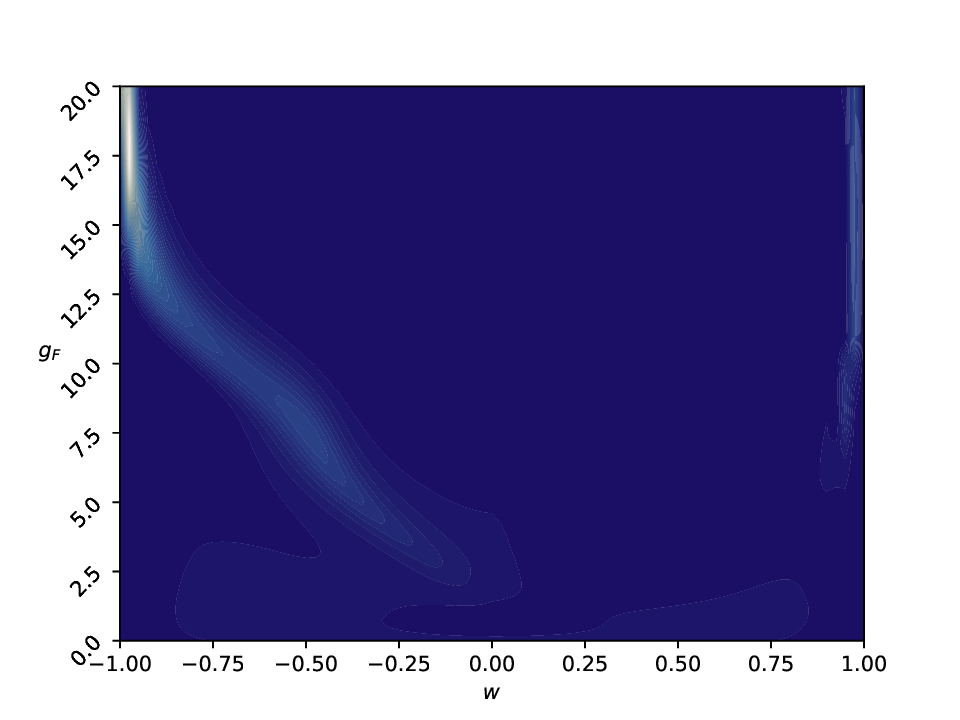}
            \caption{Follower time evolution $g_F$.}
	    \end{subfigure}
        
	    \hspace{0.25\linewidth}
        	    \begin{subfigure}[b]{0.5\linewidth}
		\centering
		    \includegraphics[scale=0.4]{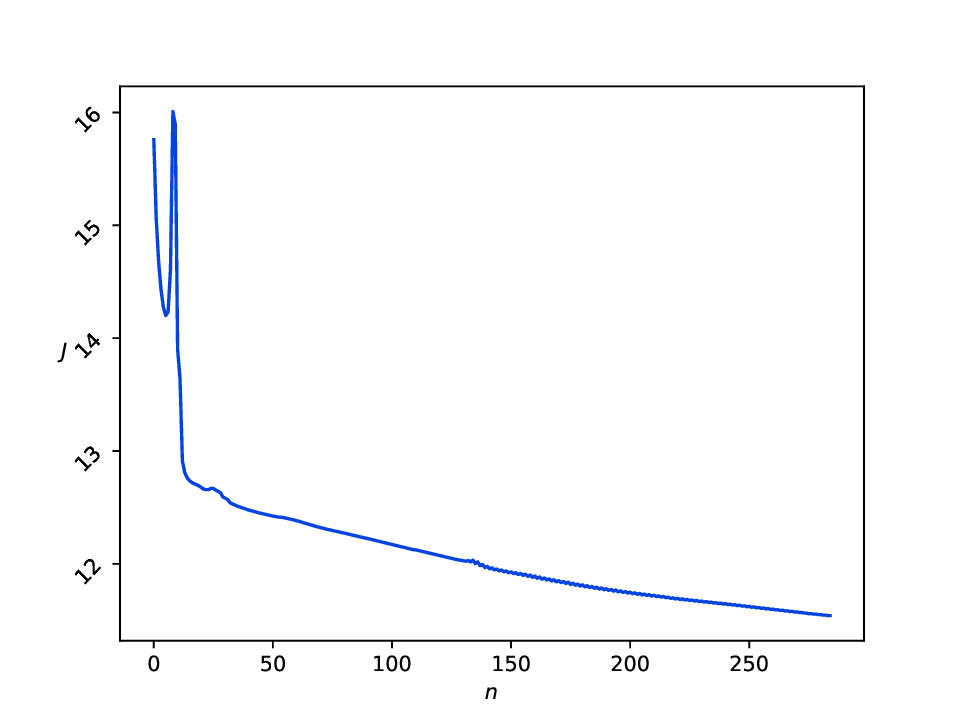}
		    \caption{Cost functional.}
            \label{centringres}
	    \end{subfigure}
        
	    \caption{The top set of graphs represent the final time states for the leader and follower species (solid blue), the uncontrolled dynamics (solid black) and the desired opinion $w_d$ according to the cost functional \eqref{centring} (dashed red). The middle set of graphs represent the top-down view of the time evolution of the leader and follower species, using the Sznajd-type interactions. The bottom graph is the values of the cost functional \eqref{centring} evaluated at each step of Algorithm~\ref{alg1}.}
	    \label{szncentringgraphs}
    \end{figure}

The results presented in Figure~\ref{szncentringgraphs} use the cost
functional \eqref{centring} and we can see that although the leader
species concentrates around $w_d=-0.5$ the follower species does not
vary from the uncontrolled dynamics by much. In these dynamics, we see
an increase in the early values for $J$ -- which is due to the uncontrolled nature of the follower species. The density of the leader species starts to approach $w_d$ at around sweeping iteration 6, however the follower density is still at both the extremes $+1$ and $-1$. As the leaders move from the positive half of the domain to the negative half, the followers that start with an opinion around $w=0$ move to the boundary at $-1$ rather than that at $+1$. This then leads to a new leader strategy to minimise in this new regime. This effect can be seen in the presented time evolution for the leader density, Figure~\ref{sznleaderdynam}, where in earlier time steps, a concentration of the leaders close to $+1$ causes the follower density to form a larger peak at the $-1$ boundary.

When using the cost functional \eqref{Fcentring} -- presented in Figure~\ref{sznFcentringgraphs} -- the leader strategy changes dramatically to cause the concentration of the followers about the desired opinion.

    \begin{figure}
	    \begin{subfigure}[b]{0.5\linewidth}
		\centering
		    \includegraphics[scale=0.4]{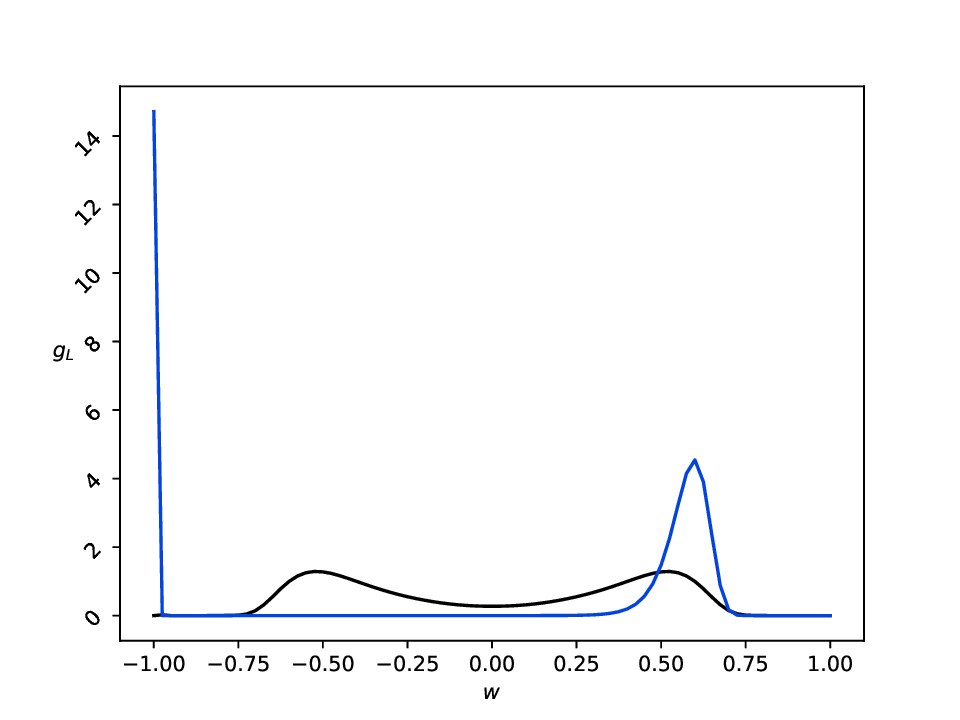}
		    \caption{Final time leader density.}
	    \end{subfigure}
		\hspace{0.04cm}
	    \begin{subfigure}[b]{0.5\linewidth}
		    \centering
		    \includegraphics[scale=0.4]{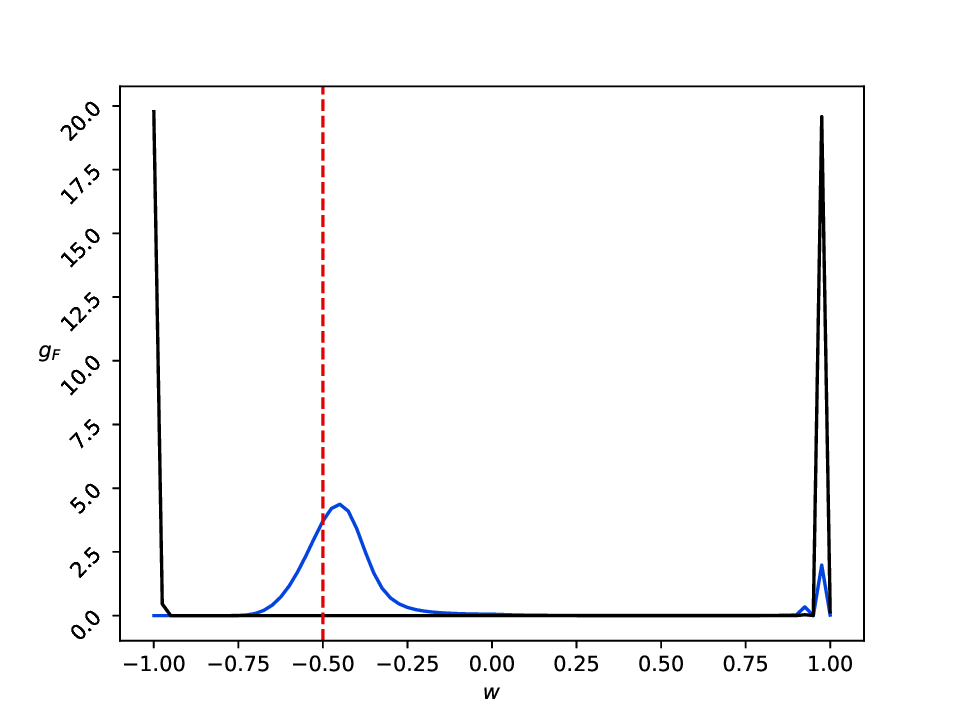}
		    \caption{Final time follower density.}
	    \end{subfigure} 
    
    	\begin{subfigure}[b]{0.5\linewidth}
    	\centering
	    	\includegraphics[scale=0.4]{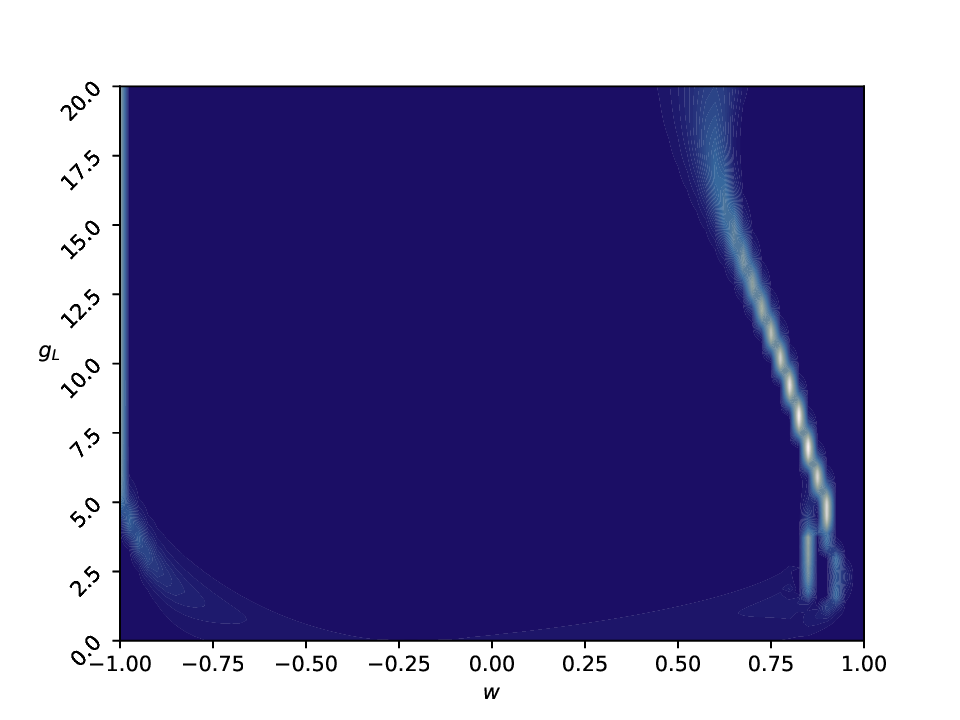}
		    \caption{Leader time evolution $g_L$.}
	    \end{subfigure} 
	    \hspace{0.04cm}
	    \begin{subfigure}[b]{0.5\linewidth}
		    \centering
		    \includegraphics[scale=0.4]{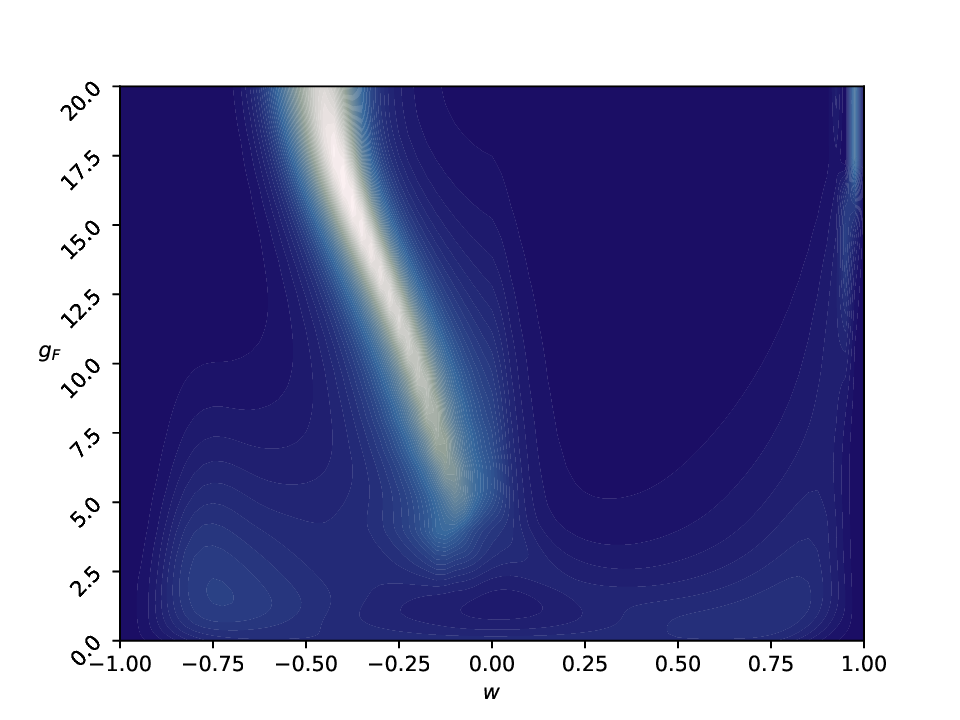}
            \caption{Follower time evolution $g_F$.}
	    \end{subfigure}
        
        	    \hspace{0.25\linewidth}
        	    \begin{subfigure}[b]{0.5\linewidth}
		\centering
		    \includegraphics[scale=0.4]{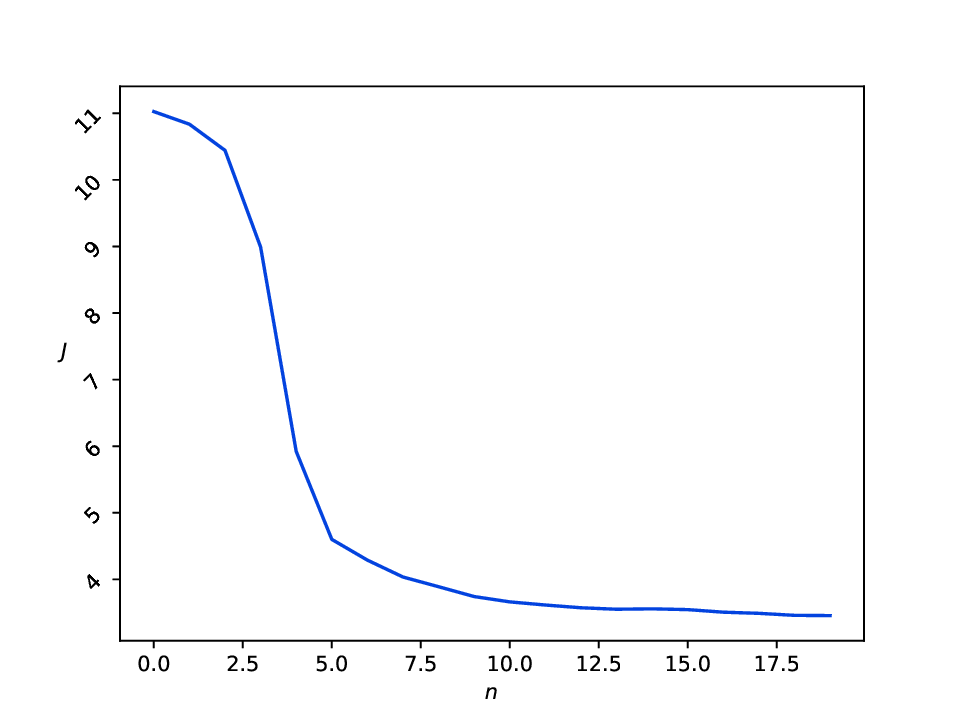}
		    \caption{Cost functional.}
            % \label{centringres}
	    \end{subfigure}
	    \caption{The top set of graphs represent the final time states for the leader and follower species (solid blue), the uncontrolled dynamics (solid black) and the desired opinion $w_d$ according to the cost functional \eqref{Fcentring} (dashed red). The middle set of graphs represent the top-down view of the time evolution of the leader and follower species using the Sznajd-type interactions. The bottom graph is the values of the cost functional \eqref{Fcentring} evaluated at each step of Algorithm~\ref{alg1}.}
	    \label{sznFcentringgraphs}
    \end{figure}

Finally, we look at the cost functional 
\begin{equation}\label{FT}
    J(\bg, u) = \frac{1}{2} \int_\I  \norm{\bg(T) - \bg_\I}^2\ \od w + \frac{\beta}{2} \int_{0}^T \int_{\I} |u|^2 \ \od w \od t,
\end{equation}
where we return to using the bounded confidence model, that is
$P_L=P_F = \Tilde{P} = P$ as in the first experiment. In the
experiment presented here, the target function is defined with
$\bg_\I(w) = \big(g_{\I}(w), g_{\I}(w)\big)$ where $g_\I$ is
constructed by setting specific values for the $w\in\{-0.5, 0, 0.5\}$
and then using a Lagrangian polynomial interpolation of these points
with a repeated root at $w=-1$ and $w=1$, i.e.\ $g_{\I}(-1) = g_{\I}(1) = 0$. The specific values for this experiment were chosen as $g_\I (-0.5) = 0.5$, $g_\I (0) = 0.25$ and $g_\I (0.5) = 1$.

    \begin{figure}
	    \begin{subfigure}[b]{0.5\linewidth}
		\centering
		    \includegraphics[scale=0.4]{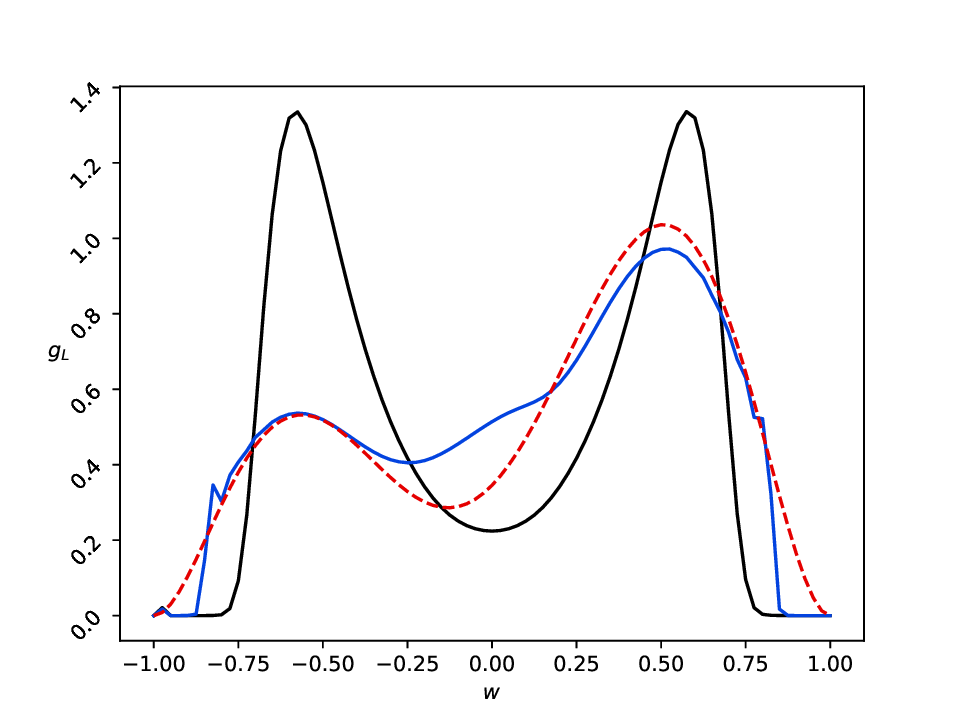}
		    \caption{Final time leader density.}
	    \end{subfigure}
		\hspace{0.04cm}
	    \begin{subfigure}[b]{0.5\linewidth}
		    \centering
		    \includegraphics[scale=0.4]{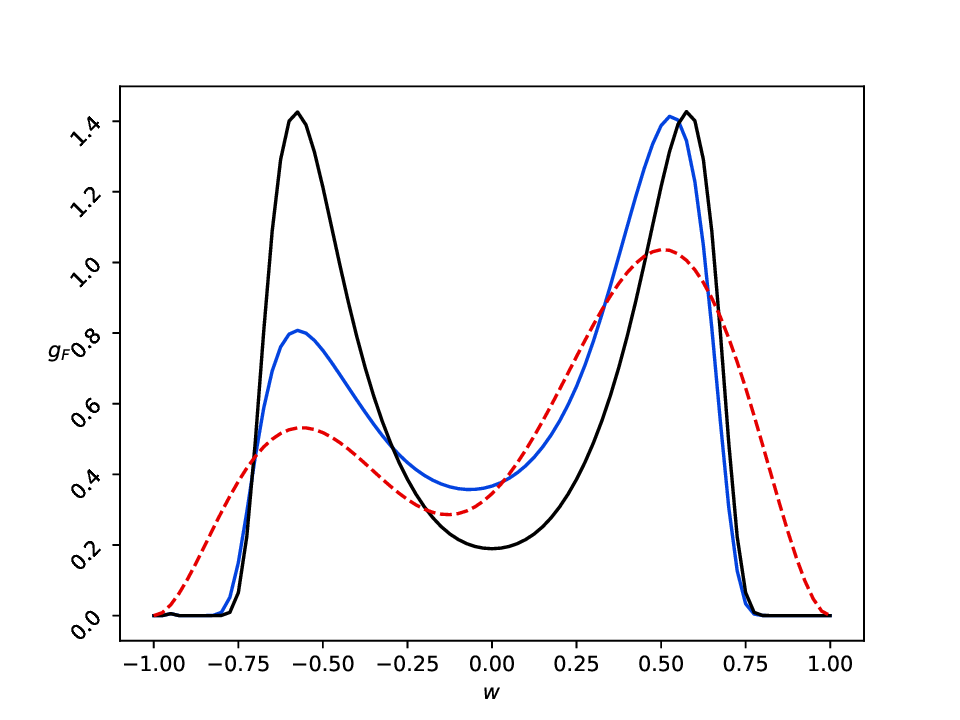}
		    \caption{Final time follower density.}
	    \end{subfigure} 
    
    	\begin{subfigure}[b]{0.5\linewidth}
    	\centering
	    	\includegraphics[scale=0.4]{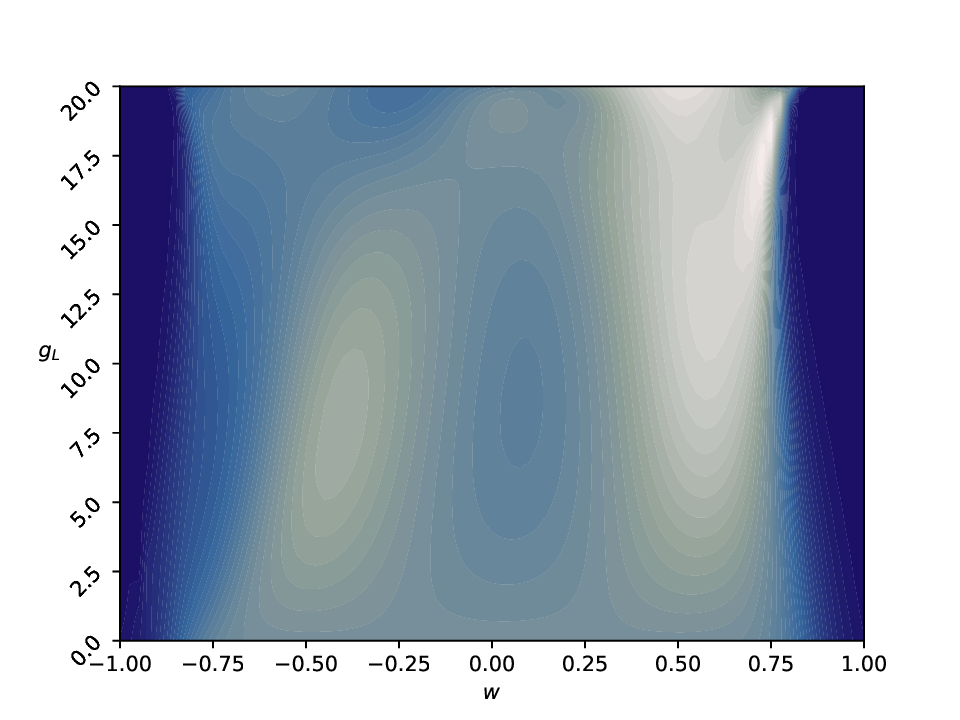}
		    \caption{Leader time evolution $g_L$.}
	    \end{subfigure} 
	    \hspace{0.04cm}
	    \begin{subfigure}[b]{0.5\linewidth}
		    \centering
		    \includegraphics[scale=0.4]{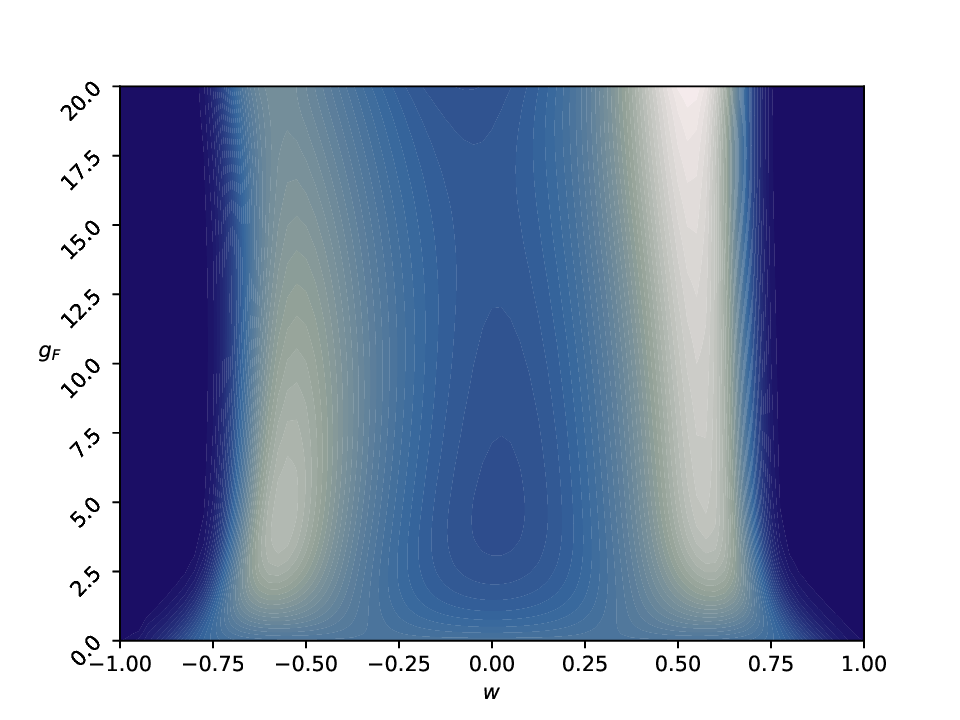}
            \caption{Follower time evolution $g_F$.}
	    \end{subfigure}
        
                	    \hspace{0.25\linewidth}
        	    \begin{subfigure}[b]{0.5\linewidth}
		\centering
		    \includegraphics[scale=0.4]{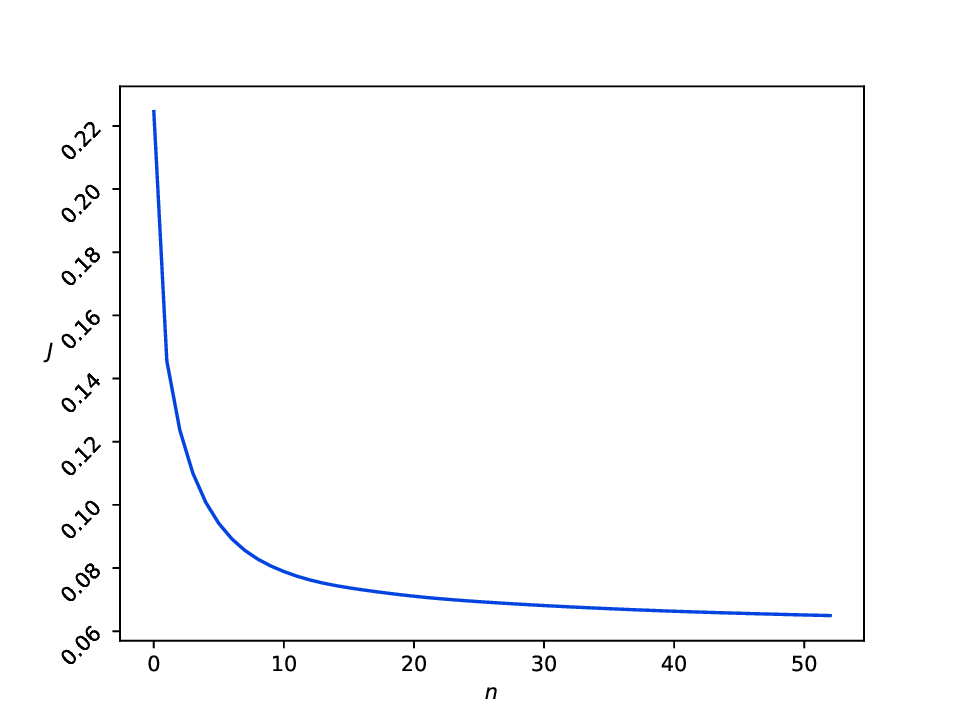}
		    \caption{Cost functional.}
            % \label{centringres}
	    \end{subfigure}
	    \caption{The top set of graphs represent the final time states for the leader and follower species (solid blue), the uncontrolled dynamics (solid black) and target state density, $\bg_\I$ according to the cost functional \eqref{FT} (dashed red). The middle set of graphs represent the top-down view of the time evolution of the leader and follower species. The bottom graph is the values of the cost functional \eqref{FT} evaluated at each step of Algorithm~\ref{alg1}.}
	    \label{FTgraphs}
    \end{figure}

As can be seen from Figure~\ref{FTgraphs}, the leader species changes dramatically from the uncontrolled dynamics to closer match the target density $g_\I$. Although a less drastic change has happened to the follower density, a clear asymmetry has appeared. Since the follower dynamics only see the control through their interactions with the leader species, we expect that the leader species would minimise the controlled dynamic, then the follower dynamic follows which reduces the total cost. In order to see a closer fit of the follower dynamics to the target, we again modify \eqref{FT} to use only the follower density,
\begin{equation}\label{FFT}
    J(\bg, u) = \frac{1}{2} \int_\I  \norm{g_F(T) - g_\I}^2\ \od w + \frac{\beta}{2} \int_{0}^T \int_{\I} |u|^2 \ \od w \od t.
\end{equation}

    \begin{figure}
	    \begin{subfigure}[b]{0.5\linewidth}
		\centering
		    \includegraphics[scale=0.4]{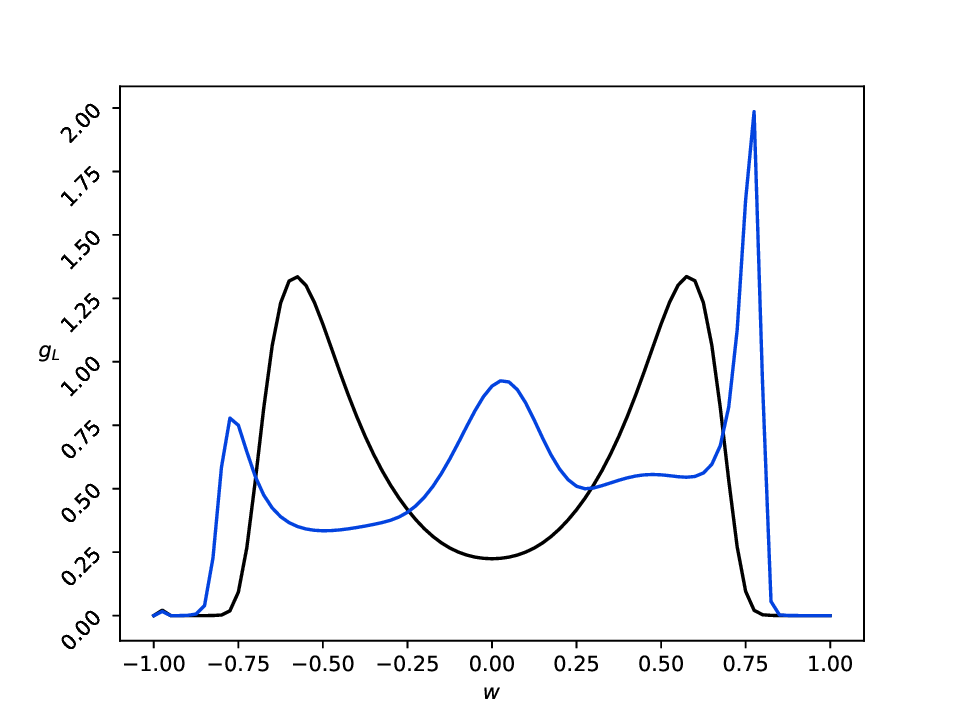}
		    \caption{Final time leader density.}
	    \end{subfigure}
		\hspace{0.04cm}
	    \begin{subfigure}[b]{0.5\linewidth}
		    \centering
		    \includegraphics[scale=0.4]{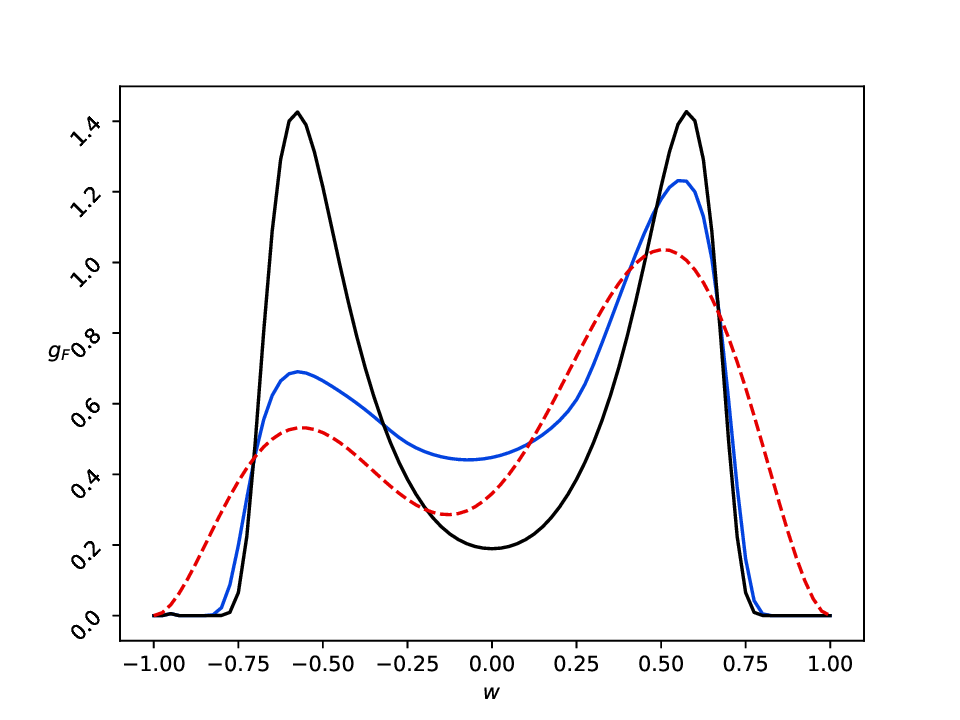}
		    \caption{Final time follower density.}
	    \end{subfigure} 
    
    	\begin{subfigure}[b]{0.5\linewidth}
    	\centering
	    	\includegraphics[scale=0.4]{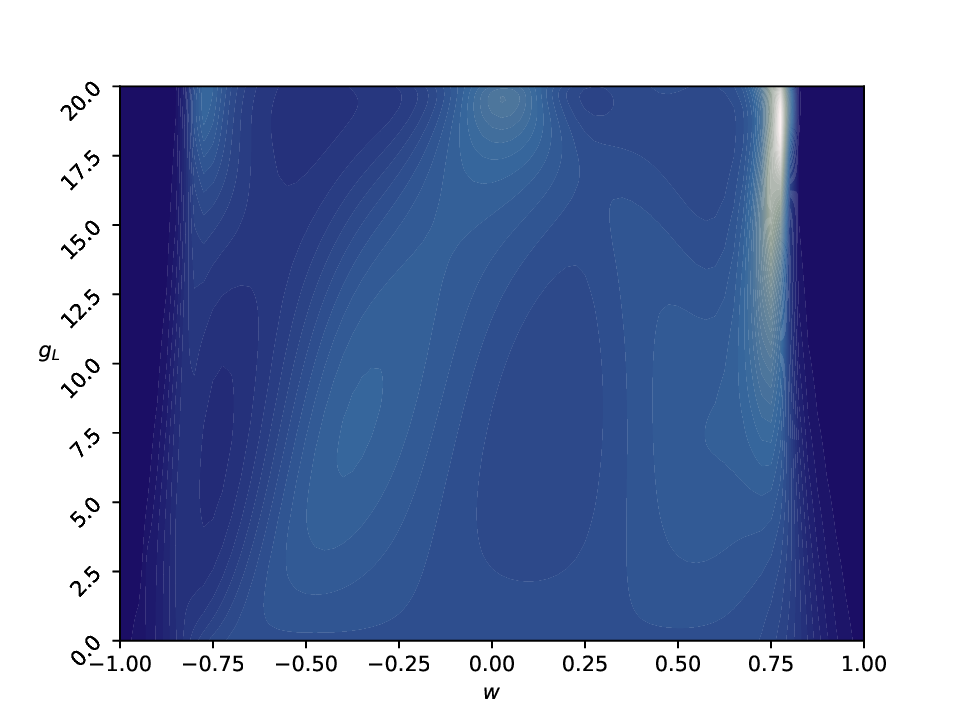}
		    \caption{Leader time evolution $g_L$.}
	    \end{subfigure} 
	    \hspace{0.04cm}
	    \begin{subfigure}[b]{0.5\linewidth}
		    \centering
		    \includegraphics[scale=0.4]{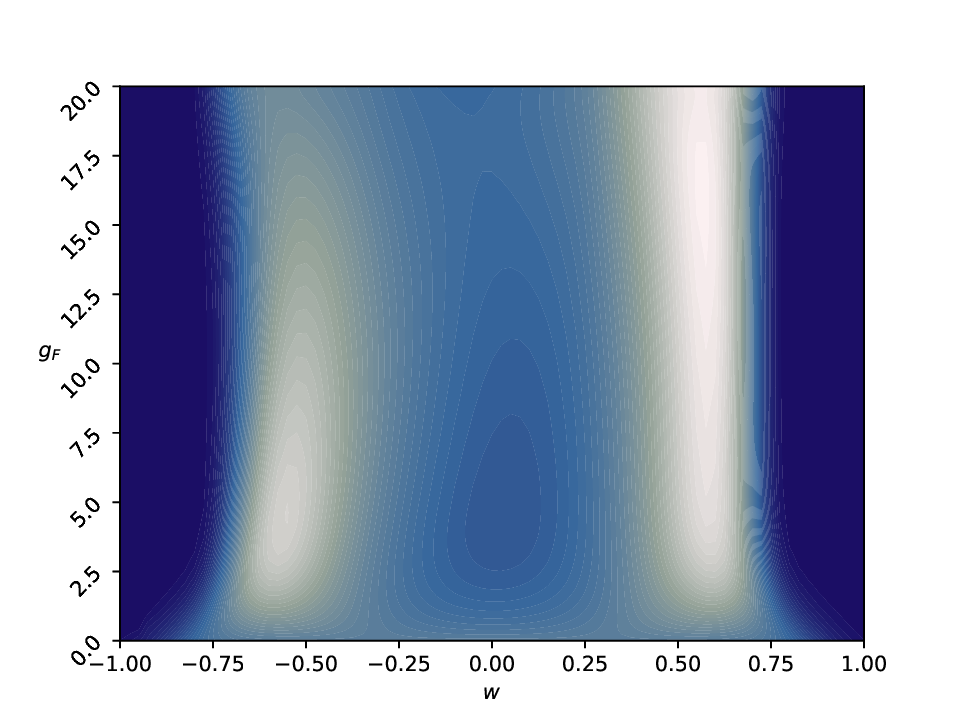}
            \caption{Follower time evolution $g_F$.}
	    \end{subfigure}
        
                        	    \hspace{0.25\linewidth}
        	    \begin{subfigure}[b]{0.5\linewidth}
		\centering
		    \includegraphics[scale=0.4]{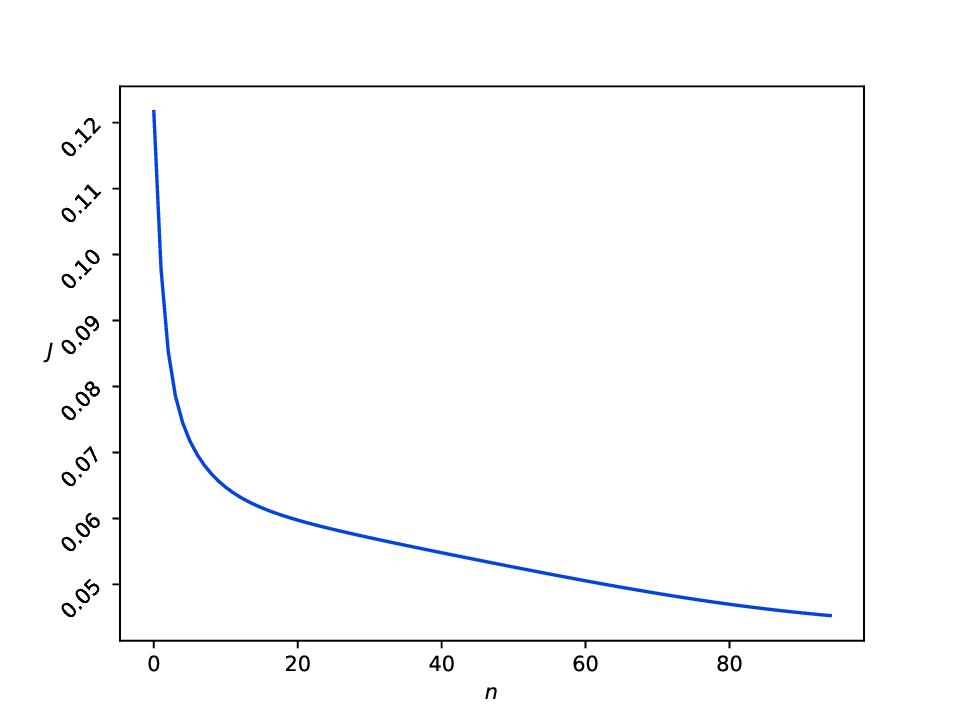}
		    \caption{Cost functional.}
            % \label{centringres}
	    \end{subfigure}
	    \caption{The top set of graphs represent the final time states for the leader and follower species (solid blue), the uncontrolled dynamics (solid black) and target state density, $\bg_\I$ according to the cost functional \eqref{FT} (dashed red). The middle set of graphs represent the top-down view of the time evolution of the leader and follower species. The bottom graph is the values of the cost functional \eqref{FFT} evaluated at each step of Algorithm~\ref{alg1}.}
	    \label{FFTgraphs}
    \end{figure}

Once again -- as presented in Figure~\ref{FFTgraphs} -- we see a large change in the leader density but a closer match of the follower dynamics to the desired final time density.

%\bmhead{Acknowledgements}

\section*{Declarations}

\begin{itemize}
\item Funding: This work was partially supported by a Royal Society International
Ex\-chan\-ges grant (Ref. IES/R3/213113). O.W.\ is supported by the Warwick
Mathematics Institute Centre for Doctoral Training in Mathematics, and gratefully acknowledges funding from the University of Warwick.
\item Conflict of interest: On behalf of all authors, the corresponding author states that there is no conflict of interest. 
\end{itemize}

\begin{appendices}

\section{Derivation of controlled Boltzmann system}\label{secA1}

We consider the evolution of the random variables that describe the opinion of an individual in the leader species $X_L$, in a the small time interval $[t, t+\Delta t]$ for some small $\Delta t$. We assume this time step is chosen such that $X_L$ interacts at most once, and the interactions that can occur will be of type \eqref{int:LL}, \eqref{int:FL} or $X_L$ does not interact in this time interval. We can now introduce the random variable, $T$, with probability mass function,
\begin{equation}\label{Boltz:pmf}
    \begin{split}
        \mathbb{P} \{ T = 0 \} &= \alpha \dt, \\
        \mathbb{P} \{ T = 1 \} &= \beta \dt, \\
        \mathbb{P} \{ T = 2 \} &= 1 -(\alpha + \beta) \dt,
    \end{split}
\end{equation}
with $\alpha, \beta > 0$ weighting the probability of interaction with the leader and follower respectively. Here $T=0$ represents the interaction \eqref{int:LL}, $T=1$ represents \eqref{int:FL} and $T=2$ represents no interaction. Labelling the other potential interacting individuals as $Y_L$ and $Y_F$ as the leader and follower respectively, we can write the opinion $X_L$ at time $t+\dt$ as,
\begin{equation} \label{Boltz:poteninteract}
        X_L(t+\Delta t) = \mb{1}_{\{0\}}(T)\, X^*_L[X_L, Y_L](t) + \mb{1}_{\{1\}}(T)\, X^*_L[X_L, Y_F](t) + \mb{1}_{\{2\}}(T)\, X_L(t),
\end{equation}
with $X^*_L[X_L, Y_L]$ representing the position of $X_L$ under the interaction \eqref{int:LL}, and $X_L^*[X_L,Y_F]$ the position of $X_L$ under the interaction \eqref{int:FL}. Note that in our case $X_L^*[X_L, Y_F] = X_L$.

Now considering $\mnphi{X_L(t + \dt)}$ -- where $\mean{\cdot}$ is the mean with respect to $X_L$, $Y_L$, $Y_F$ and $\eta_L$ and $\phi \in C_c^{2,\delta}(\I;\real)$ -- and using \eqref{Boltz:pmf}, we obtain,
\begin{equation*}
    \begin{split}
        \mnphi{X_L(t + \dt} =& \bigg\langle \phi\bigg( \mb{1}_{\{0\}}(T)\, X^*_L[X_L, Y_L](t) + \mb{1}_{\{1\}}(T)\,X^*_L[X_L, Y_F](t)\\
        \quad &+ \mb{1}_{\{2\}}(T)\, X_L(t)\bigg\rangle, \\[5pt]
        =& \alpha \dt \mnphi{X_L^*[X_L, Y_L]} + \beta \dt \mnphi{X_L^*[X_L,Y_F]} \\
        &\quad \mnphi{X_L} - (\alpha + \beta)\dt \mnphi{X_L},
    \end{split}
\end{equation*}
through straightforward calculation, we get,
\begin{equation*}
\begin{split}
    \frac{\mean{\phi(X_L(t + \dt) - \phi(X_L(t))}}{\dt} &= \alpha \mean{\phi(X_L^*[X_L, Y_L]) - \phi(X_L)} \\
    & \quad + \beta \mean{\phi(X_L^*[X_L,X_i]) - \phi(X_L)},
\end{split}
\end{equation*}
we now apply the limit as $\dt \rightarrow 0$,
\begin{equation} \label{Boltz:leadx} 
    \deriv{\mnphi{X_L}}{t} = \varepsilon \mean{\phi(X_L^*[X_L, Y_L]) - \phi(X_L)} + \beta \mean{\phi(X_L^*[X_L,Y_F]) - \phi(X_L)}.
\end{equation}
Note here that the individual $Y_L$ is also a member of the leader species and thus the density which describes the distribution of $X_L$ also describes that of $Y_L$. Writing the expectations of $X_L$, $Y_L$ and $Y_F$ explicitly yields the weak form of the Boltzmann-type equation for the leader species,
\begin{equation} \label{Boltz:weaklead}
    \begin{split}
        \nderiv{t}\left(\int_{\I} \phi(w) f_L(w,t) \ \od w\right) =& \alpha\Big \langle \int_{\I^2}\big[ \phi(w^*) + \phi(v^*) - \phi(w) - \phi(v) \big] \\
        &\qquad \qquad \times  f_L(w,t) f_L(v,t) \ \od w \od v \Big \rangle \\
        &+\beta \Big \langle \int_{\I^2} \big[ \phi(w^*) - \phi(w)\big] \\
        &\qquad \qquad\times f_L(w,t) f_F(v,t) \ \od w \od v \Big \rangle.
    \end{split}
\end{equation}
Notice that the second term on the right hand side of \eqref{Boltz:weaklead} will cancel out, since the post-interaction opinion of the leader with a follower results in no change to the leader opinion, that is $\phi(w^*) - \phi(w) = 0$. We can then write the strong equation,
\begin{equation} \label{Boltz:stronglead}
    \npderiv{t} f_L(w,t) = \frac{\alpha}{2} Q_L[f_L, f_L](w,t),
\end{equation}
defining the $Q_L$ operator as the strong form of the remaining term on right hand side of \eqref{Boltz:weaklead}.

Using a slight modification of the above method for the follower species, we derive the remaining equations of the system,
\begin{equation}
\begin{split}
    \deriv{\mnphi{X_F}}{t} &= a \mean{\phi(X_F^*[X_F, Y_L]) - \phi(X_F)} \\
    & \qquad + b \mean{\phi(X_F^*[X_F,Y_F]) - \phi(X_F)}, \label{Boltz:followx}
\end{split}
\end{equation}
Once again, since $X_F$ and $Y_F$ are both members of the follower species, the densities which describe their distributions are the same, writing the expectation in \eqref{Boltz:followx} explicitly, we obtain the weak form of the Boltzmann-type equation for the follower species,
\begin{equation} \label{Boltz:weakfollow}
    \begin{split}
        \nderiv{t}\left(\int_{\I} \phi(w) f_F(w,t) \ \od w\right) =& a\Big \langle \int_{\I^2}\big[ \phi(w^*) + \phi(v^*) - \phi(w) - \phi(v) \big] \\
        & \qquad \qquad \times f_F(w,t) f_F(v,t) \ \od w \od v \Big \rangle \\
        &+ b \Big \langle \int_{\I^2} \big[ \phi(w^*) - \phi(w)\big] \\
        & \qquad \qquad \times f_F(w,t) f_L(v,t) \ \od w \od v \Big \rangle,
    \end{split}
\end{equation}
with the strong form of this equation,
    \begin{equation}\label{Boltz:strongfollow}
        \npderiv{t}f_F(w, t) = \frac{a}{2} Q_{F}[f_F, f_F](w,t) + \frac{b}{2} Q_{FL}[f_F, f_{FL}](w,t).
    \end{equation}
    The interaction operators $Q_L$, $Q_{FL}$, and $Q_F$ in \eqref{Boltz:stronglead} and \eqref{Boltz:strongfollow}, are expressed in weak form as,
    \begin{align*}
        \int_{\I} Q_L[f_L, f_L](w,t) \phi(w) \ \od w =& \Big \langle \int_{\I^2} \big[ \phi(w^*) + \phi(v^*) - \phi(w) - \phi(v)\big] \\
        & \qquad \qquad \times f_L(w,t) f_L(v,t) \ \od w \od v \Big \rangle, \\
        \int_{\I} Q_{FL}[f_F, f_L](w,t) \phi(w) \ \od w =& \Big \langle \int_{\I^2} \big[ \phi(w^*) - \phi(w)\big]  f_F(w,t) f_L(v,t) \ \od w \od v \Big \rangle,\\
        \int_{\I} Q_F[f_F, f_F](w,t) \phi(w) \ \od w =& \Big \langle \int_{\I^2} \big[ \phi(w^*) + \phi(v^*) - \phi(w) - \phi(v) \big] \\
        & \qquad \qquad \times f_F(w,t) f_F(v,t) \ \od w \od v \Big \rangle,
    \end{align*}
    where $\phi$ are test functions from $C_c^{2,\delta}(\I;\real)$, that are continuous and compactly supported functions on $\I$. In \eqref{Boltz:stronglead} and \eqref{Boltz:strongfollow}, the values of $\alpha$, $a$, $b$ denote the interaction frequencies, which we can write as relaxation times of the Boltzmann-type equations using $\alpha = 1/\tau_{LL}$, $a = 1/\tau_{FF}$ and $b = 1/\tau_{FL}$.

\section{Derivation of controlled Fokker-Planck system}\label{secA2}
 We derive the coupled system of non-linear, non-local Fokker-Planck equations for this model, starting with the leader equation by considering the first order Taylor's expansion of $\phi(w^*)$ about $w$,
    \begin{equation} \label{Fokker:taylor} 
    \begin{split}
        \mean{\phi(w^*)- \phi(w)} =& - \gamma_L \phi'(w) P_{LL}(v,w)(w-v) + \frac{\phi'' (w)}{2} D^2(w) \sigma^2 \\ 
        &+\mean{\zeta(w,v,t) + R(w,v,t)}.
    \end{split}
    \end{equation}
    with $\zeta$ and $R$ defined as,
    \begin{align*}
        \zeta(w,v,t) =& \tfrac{1}{2} \phi '' (w) \bigg(\gamma_L^2 P_{LL}^2(v,w)(w-v)^2 + \frac{\gamma_L^2}{4} u^2(w,t) + \frac{\gamma_L^2}{2} P_{LL}(v,w) (w-v) u(w, t)\bigg), \\
        R(w,v,t) =& \frac{1}{2} \bigg(\phi '' (\Tilde{w})(w^*-w)^2 - \phi '' (w)(w^*-w)^2\bigg).
    \end{align*}
    where $\Tilde{w} = \kappa w + (1-\kappa)w^*$ for $\kappa \in [0,1]$. Applying the same Taylor's expansion for $\phi(v^*)$ about $v$, substituting into $Q_L$ and then iterating the integral on the right hand side yields,
\begin{equation*} 
    \begin{split}
        \int_{\I} Q_L[f_L, f_L](w,t) \phi(w) \ \od w =& -2 \gamma_L \int_\I \phi ' (w) \mc{K}[f_L](w,t) f_L(w,t) \ \od w \\
        & - \gamma_L \int_\I \phi ' (w) u(w,t) f_L(w,t) \ \od w \\
        &+ 2\sigma^2 \int_\I \frac{\phi '' (w)}{2} D^2(w) f_L(w,t) \ \od w \\
        &+ 2 \int_\I{\mean{\zeta(w,v,t) + R(w,v,t)}} \\
        & \qquad \qquad \times f_L(w,t) f_L(v,t) \ \od w \od v,
    \end{split}
\end{equation*}
    where,
    \begin{equation*}
        \mc{K}[f_L](w,s) := \int_\I P_L(w,v)(w-v)f_L(v,s) \ \od v.
    \end{equation*}
    Now considering the time-scaling $s=\gamma_L t$ and the
    transformed density functions $g_L(w,s) = f_L(w,t)$ and taking the
    quasi-invariant limit of \eqref{Boltz:stronglead} -- that is
    $\sigma, \gamma_L \rightarrow 0$ while keeping the value
    $\sigma^2/\gamma_L$ constant throughout.
Notice that under the time scaling and quasi-invariant limit, $R$
    goes to zero (see \cite{bib:During:strongleaders}) and in the $\zeta$ term, we still have a factor of $\gamma_L$ even after dividing through by $\gamma_L$ thus this term also goes to zero.
    We then write
    the strong form of the Fokker-Planck-type equation for the leader species,
    \begin{equation*} %\label{Fokker:lead}
        \begin{split}
            \pderiv{g_L}{s}(w, s) =& \npderiv{w} \left( \left( \frac{1}{\tau_{LL}} \mc{K}[g_L](w,s) + \frac{1}{2\tau_{LL}} u(w,s)\right) g_L(w,s) \right)\\
            &+ \frac{\lambda_{L} }{2\tau_{LL}}\nsecpderiv{w} (D^2(w) g_L(w,s)),
        \end{split}
      \end{equation*}
      with $\lambda_L =
      M_L\sigma^2/\gamma_L$ and $M_L=\int_{\I} f_L(w,t)\,dw$, subject to no-flow boundary conditions.

    Applying the same process for \eqref{Boltz:strongfollow}, with the
    scaling $\gamma_L = \alpha_{LF} \gamma_F$ and letting $\lambda_F =
      M_F\sigma^2/\gamma_F$ with $M_F=\int_{\I} f_F(w,t)\,dw$, yields the follower Fokker-Planck-type equation,
    \begin{equation*} %\label{Fokker:follow}
        \begin{split}
            \pderiv{g_F}{s}(w,s) =& \alpha_{LF} \npderiv{w} \left(\left(\frac{1}{2\tau_{FL}} \mc{M}[g_L](w,s)+ \frac{1}{\tau_{FF}}\mc{N}[g_F](w,s) \right)g_F (w,s)\right) \\
            &+  \left(\frac{\lambda_{F}}{4\tau_{LF}} + \frac{\lambda_{F}}{2\tau_{FF}}\right) \nsecpderiv{w} (D^2(w) g_F(w,s)),    
        \end{split}
    \end{equation*}
     subject to no-flow boundary conditions, where:
    \begin{align*}
         \mc{M}[g_L](w,s) :=& \int_\I \Tilde{P}(w,v)(w-v)g_L(v,s) \ \od v,\\
         \mc{N}[g_F](w,s) :=& \int_\I P_F(w,v)(w-v)g_F(v,s) \ \od v.
    \end{align*}

\end{appendices}

\bibliography{bdbib}% common bib file


\begin{thebibliography}{27}
\ifx \bisbn   \undefined \def \bisbn  #1{ISBN #1}\fi
\ifx \binits  \undefined \def \binits#1{#1}\fi
\ifx \bauthor  \undefined \def \bauthor#1{#1}\fi
\ifx \batitle  \undefined \def \batitle#1{#1}\fi
\ifx \bjtitle  \undefined \def \bjtitle#1{#1}\fi
\ifx \bvolume  \undefined \def \bvolume#1{\textbf{#1}}\fi
\ifx \byear  \undefined \def \byear#1{#1}\fi
\ifx \bissue  \undefined \def \bissue#1{#1}\fi
\ifx \bfpage  \undefined \def \bfpage#1{#1}\fi
\ifx \blpage  \undefined \def \blpage #1{#1}\fi
\ifx \burl  \undefined \def \burl#1{\textsf{#1}}\fi
\ifx \doiurl  \undefined \def \doiurl#1{\url{https://doi.org/#1}}\fi
\ifx \betal  \undefined \def \betal{\textit{et al.}}\fi
\ifx \binstitute  \undefined \def \binstitute#1{#1}\fi
\ifx \binstitutionaled  \undefined \def \binstitutionaled#1{#1}\fi
\ifx \bctitle  \undefined \def \bctitle#1{#1}\fi
\ifx \beditor  \undefined \def \beditor#1{#1}\fi
\ifx \bpublisher  \undefined \def \bpublisher#1{#1}\fi
\ifx \bbtitle  \undefined \def \bbtitle#1{#1}\fi
\ifx \bedition  \undefined \def \bedition#1{#1}\fi
\ifx \bseriesno  \undefined \def \bseriesno#1{#1}\fi
\ifx \blocation  \undefined \def \blocation#1{#1}\fi
\ifx \bsertitle  \undefined \def \bsertitle#1{#1}\fi
\ifx \bsnm \undefined \def \bsnm#1{#1}\fi
\ifx \bsuffix \undefined \def \bsuffix#1{#1}\fi
\ifx \bparticle \undefined \def \bparticle#1{#1}\fi
\ifx \barticle \undefined \def \barticle#1{#1}\fi
\bibcommenthead
\ifx \bconfdate \undefined \def \bconfdate #1{#1}\fi
\ifx \botherref \undefined \def \botherref #1{#1}\fi
\ifx \url \undefined \def \url#1{\textsf{#1}}\fi
\ifx \bchapter \undefined \def \bchapter#1{#1}\fi
\ifx \bbook \undefined \def \bbook#1{#1}\fi
\ifx \bcomment \undefined \def \bcomment#1{#1}\fi
\ifx \oauthor \undefined \def \oauthor#1{#1}\fi
\ifx \citeauthoryear \undefined \def \citeauthoryear#1{#1}\fi
\ifx \endbibitem  \undefined \def \endbibitem {}\fi
\ifx \bconflocation  \undefined \def \bconflocation#1{#1}\fi
\ifx \arxivurl  \undefined \def \arxivurl#1{\textsf{#1}}\fi
\csname PreBibitemsHook\endcsname

%%% 1
\bibitem[\protect\citeauthoryear{Düring
  et~al.}{2009}]{bib:During:strongleaders}
\begin{barticle}
\bauthor{\bsnm{Düring}, \binits{B.}},
\bauthor{\bsnm{Markowich}, \binits{P.}},
\bauthor{\bsnm{Pietschmann}, \binits{J.}},
\bauthor{\bsnm{Wolfram}, \binits{M.}}:
\batitle{{B}oltzmann and {F}okker-{P}lanck equations modelling opinion
  formation in the presence of strong leaders}.
\bjtitle{Proc. R. Soc. A.}
\bvolume{465},
\bfpage{3687}--\blpage{3708}
(\byear{2009})
\doiurl{10.1098/rspa.2009.0239}
\end{barticle}
\endbibitem

%%% 2
\bibitem[\protect\citeauthoryear{Fornasier and
  Solombrino}{2014}]{fornasier2014mean}
\begin{barticle}
\bauthor{\bsnm{Fornasier}, \binits{M.}},
\bauthor{\bsnm{Solombrino}, \binits{F.}}:
\batitle{Mean-field optimal control}.
\bjtitle{ESAIM: Control, Optimisation and Calculus of Variations}
\bvolume{20}(\bissue{4}),
\bfpage{1123}--\blpage{1152}
(\byear{2014})
\doiurl{10.1051/cocv/2014009}
\end{barticle}
\endbibitem

%%% 3
\bibitem[\protect\citeauthoryear{Albi et~al.}{2017}]{bib:Albi:meancontrol}
\begin{barticle}
\bauthor{\bsnm{Albi}, \binits{G.}},
\bauthor{\bsnm{Choi}, \binits{Y.-P.}},
\bauthor{\bsnm{Fornasier}, \binits{M.}},
\bauthor{\bsnm{Kalise}, \binits{D.}}:
\batitle{Mean field control hierarchy}.
\bjtitle{Applied Mathematics \& Optimization}
\bvolume{76}(\bissue{1}),
\bfpage{93}--\blpage{135}
(\byear{2017})
\doiurl{10.1007/s00245-017-9429-x}
\end{barticle}
\endbibitem

%%% 4
\bibitem[\protect\citeauthoryear{Bertotti and
  Delitala}{2008}]{bertotti2008discrete}
\begin{barticle}
\bauthor{\bsnm{Bertotti}, \binits{M.L.}},
\bauthor{\bsnm{Delitala}, \binits{M.}}:
\batitle{On a discrete generalized kinetic approach for modelling persuader’s
  influence in opinion formation processes}.
\bjtitle{Math. comput. model.}
\bvolume{48}(\bissue{7-8}),
\bfpage{1107}--\blpage{1121}
(\byear{2008})
\doiurl{10.1016/j.mcm.2007.12.021}
\end{barticle}
\endbibitem

%%% 5
\bibitem[\protect\citeauthoryear{Düring and
  Wolfram}{2015}]{bib:During:inhomogeneous}
\begin{barticle}
\bauthor{\bsnm{Düring}, \binits{B.}},
\bauthor{\bsnm{Wolfram}, \binits{M.}}:
\batitle{Opinion dynamics: inhomogeneous {B}oltzmann-type equations modelling
  opinion leadership and political segregation}.
\bjtitle{Proc. R. Soc. A.}
\bvolume{471},
\bfpage{20150345}
(\byear{2015})
\doiurl{10.1098/rspa.2015.0345}
\end{barticle}
\endbibitem

%%% 6
\bibitem[\protect\citeauthoryear{Hinze et~al.}{2008}]{hinze2008optimization}
\begin{bbook}
\bauthor{\bsnm{Hinze}, \binits{M.}},
\bauthor{\bsnm{Pinnau}, \binits{R.}},
\bauthor{\bsnm{Ulbrich}, \binits{M.}},
\bauthor{\bsnm{Ulbrich}, \binits{S.}}:
\bbtitle{Optimization with PDE Constraints}.
\bsertitle{Mathematical Modelling: Theory and Applications},
vol. \bseriesno{23}.
\bpublisher{Springer},
\blocation{Berlin, Heidelberg}
(\byear{2008}).
\doiurl{10.1007/978-1-4020-8839-1}
\end{bbook}
\endbibitem

%%% 7
\bibitem[\protect\citeauthoryear{Tr{\"o}ltzsch}{2010}]{troltzsch2010optimal}
\begin{bbook}
\bauthor{\bsnm{Tr{\"o}ltzsch}, \binits{F.}}:
\bbtitle{Optimal Control of Partial Differential Equations: Theory, Methods,
  and Applications}.
\bsertitle{Graduate Studies in Mathematics},
vol. \bseriesno{112}.
\bpublisher{American Mathematical Soc.},
\blocation{Providence, Rhode Island}
(\byear{2010}).
\doiurl{10.1090/gsm/112}
\end{bbook}
\endbibitem

%%% 8
\bibitem[\protect\citeauthoryear{Albi
  et~al.}{2014}]{bib:AlbiPareschiZanella:MPC}
\begin{barticle}
\bauthor{\bsnm{Albi}, \binits{G.}},
\bauthor{\bsnm{Pareschi}, \binits{L.}},
\bauthor{\bsnm{Zanella}, \binits{M.}}:
\batitle{Boltzmann-type control of opinion consensus through leaders}.
\bjtitle{Phil. Trans. R. Soc. A}
(\byear{2014})
\doiurl{10.1098/rsta.2014.0138}
\end{barticle}
\endbibitem

%%% 9
\bibitem[\protect\citeauthoryear{Wongkaew et~al.}{2015}]{wongkaew2015control}
\begin{barticle}
\bauthor{\bsnm{Wongkaew}, \binits{S.}},
\bauthor{\bsnm{Caponigro}, \binits{M.}},
\bauthor{\bsnm{Borzi}, \binits{A.}}:
\batitle{On the control through leadership of the {H}egselmann--{K}rause
  opinion formation model}.
\bjtitle{Mathematical Models and Methods in Applied Sciences}
\bvolume{25}(\bissue{03}),
\bfpage{565}--\blpage{585}
(\byear{2015})
\doiurl{10.1142/S0218202515400060}
\end{barticle}
\endbibitem

%%% 10
\bibitem[\protect\citeauthoryear{Borzi and Wongkaew}{2015}]{borzi2015modeling}
\begin{barticle}
\bauthor{\bsnm{Borzi}, \binits{A.}},
\bauthor{\bsnm{Wongkaew}, \binits{S.}}:
\batitle{Modeling and control through leadership of a refined flocking system}.
\bjtitle{Mathematical Models and Methods in Applied Sciences}
\bvolume{25}(\bissue{02}),
\bfpage{255}--\blpage{282}
(\byear{2015})
\doiurl{10.1142/S0218202515500098}
\end{barticle}
\endbibitem

%%% 11
\bibitem[\protect\citeauthoryear{Bailo et~al.}{2018}]{bailo2018optimal}
\begin{barticle}
\bauthor{\bsnm{Bailo}, \binits{R.}},
\bauthor{\bsnm{Bongini}, \binits{M.}},
\bauthor{\bsnm{Carrillo}, \binits{J.A.}},
\bauthor{\bsnm{Kalise}, \binits{D.}}:
\batitle{Optimal consensus control of the cucker-smale model}.
\bjtitle{IFAC-PapersOnLine}
\bvolume{51}(\bissue{13}),
\bfpage{1}--\blpage{6}
(\byear{2018})
\doiurl{10.1016/j.ifacol.2018.07.245}
\end{barticle}
\endbibitem

%%% 12
\bibitem[\protect\citeauthoryear{Albi et~al.}{2020}]{albi2020mathematical}
\begin{bchapter}
\bauthor{\bsnm{Albi}, \binits{G.}},
\bauthor{\bsnm{Cristiani}, \binits{E.}},
\bauthor{\bsnm{Pareschi}, \binits{L.}},
\bauthor{\bsnm{Peri}, \binits{D.}}:
\bctitle{Mathematical models and methods for crowd dynamics control}.
In: \bbtitle{Crowd Dynamics, Volume 2: Theory, Models, and Applications},
pp. \bfpage{159}--\blpage{197}.
\bpublisher{Springer},
\blocation{Berlin, Heidelberg}
(\byear{2020}).
\doiurl{10.1007/978-3-030-50450-2_8}
\end{bchapter}
\endbibitem

%%% 13
\bibitem[\protect\citeauthoryear{Burger et~al.}{2020}]{burger2020instantaneous}
\begin{barticle}
\bauthor{\bsnm{Burger}, \binits{M.}},
\bauthor{\bsnm{Pinnau}, \binits{R.}},
\bauthor{\bsnm{Totzeck}, \binits{C.}},
\bauthor{\bsnm{Tse}, \binits{O.}},
\bauthor{\bsnm{Roth}, \binits{A.}}:
\batitle{Instantaneous control of interacting particle systems in the
  mean-field limit}.
\bjtitle{Journal of Computational Physics}
\bvolume{405},
\bfpage{109181}
(\byear{2020})
\doiurl{10.1016/j.jcp.2019.109181}
\end{barticle}
\endbibitem

%%% 14
\bibitem[\protect\citeauthoryear{Gong et~al.}{2023}]{gong2023crowd}
\begin{barticle}
\bauthor{\bsnm{Gong}, \binits{X.}},
\bauthor{\bsnm{Herty}, \binits{M.}},
\bauthor{\bsnm{Piccoli}, \binits{B.}},
\bauthor{\bsnm{Visconti}, \binits{G.}}:
\batitle{Crowd dynamics: Modeling and control of multiagent systems}.
\bjtitle{Annual Review of Control, Robotics, and Autonomous Systems}
\bvolume{6}(\bissue{1}),
\bfpage{261}--\blpage{282}
(\byear{2023})
\doiurl{10.1146/annurev-control-060822-123629}
\end{barticle}
\endbibitem

%%% 15
\bibitem[\protect\citeauthoryear{Albi et~al.}{2021}]{albi2021control}
\begin{barticle}
\bauthor{\bsnm{Albi}, \binits{G.}},
\bauthor{\bsnm{Pareschi}, \binits{L.}},
\bauthor{\bsnm{Zanella}, \binits{M.}}:
\batitle{Control with uncertain data of socially structured compartmental
  epidemic models}.
\bjtitle{Journal of Mathematical Biology}
\bvolume{82}(\bissue{7}),
\bfpage{63}
(\byear{2021})
\doiurl{10.1007/s00285-021-01617-y}
\end{barticle}
\endbibitem

%%% 16
\bibitem[\protect\citeauthoryear{Zanella}{2023}]{zanella2023kinetic}
\begin{barticle}
\bauthor{\bsnm{Zanella}, \binits{M.}}:
\batitle{Kinetic models for epidemic dynamics in the presence of opinion
  polarization}.
\bjtitle{Bulletin of Mathematical Biology}
\bvolume{85}(\bissue{5}),
\bfpage{36}
(\byear{2023})
\doiurl{10.1007/s11538-023-01147-2}
\end{barticle}
\endbibitem

%%% 17
\bibitem[\protect\citeauthoryear{Bondesan et~al.}{2024}]{bondesan2024kinetic}
\begin{barticle}
\bauthor{\bsnm{Bondesan}, \binits{A.}},
\bauthor{\bsnm{Toscani}, \binits{G.}},
\bauthor{\bsnm{Zanella}, \binits{M.}}:
\batitle{Kinetic compartmental models driven by opinion dynamics: vaccine
  hesitancy and social influence}.
\bjtitle{Mathematical Models and Methods in Applied Sciences}
\bvolume{34}(\bissue{06}),
\bfpage{1043}--\blpage{1076}
(\byear{2024})
\doiurl{10.1142/S0218202524400062}
\end{barticle}
\endbibitem

%%% 18
\bibitem[\protect\citeauthoryear{D{\"u}ring et~al.}{2024}]{during2024breaking}
\begin{barticle}
\bauthor{\bsnm{D{\"u}ring}, \binits{B.}},
\bauthor{\bsnm{Franceschi}, \binits{J.}},
\bauthor{\bsnm{Wolfram}, \binits{M.-T.}},
\bauthor{\bsnm{Zanella}, \binits{M.}}:
\batitle{Breaking consensus in kinetic opinion formation models on graphons}.
\bjtitle{Journal of Nonlinear Science}
\bvolume{34}(\bissue{4}),
\bfpage{79}
(\byear{2024})
\doiurl{10.1007/s00332-024-10060-4}
\end{barticle}
\endbibitem

%%% 19
\bibitem[\protect\citeauthoryear{Herty et~al.}{2007}]{herty2007instantaneous}
\begin{barticle}
\bauthor{\bsnm{Herty}, \binits{M.}},
\bauthor{\bsnm{Kirchner}, \binits{C.}},
\bauthor{\bsnm{Klar}, \binits{A.}}:
\batitle{Instantaneous control for traffic flow}.
\bjtitle{Mathematical methods in the applied sciences}
\bvolume{30}(\bissue{2}),
\bfpage{153}--\blpage{169}
(\byear{2007})
\doiurl{10.1002/mma.779}
\end{barticle}
\endbibitem

%%% 20
\bibitem[\protect\citeauthoryear{Carrillo
  et~al.}{2018}]{carrillo2018analytical}
\begin{barticle}
\bauthor{\bsnm{Carrillo}, \binits{J.A.}},
\bauthor{\bsnm{Choi}, \binits{Y.-P.}},
\bauthor{\bsnm{Totzeck}, \binits{C.}},
\bauthor{\bsnm{Tse}, \binits{O.}}:
\batitle{An analytical framework for consensus-based global optimization
  method}.
\bjtitle{Mathematical Models and Methods in Applied Sciences}
\bvolume{28}(\bissue{06}),
\bfpage{1037}--\blpage{1066}
(\byear{2018})
\doiurl{10.1142/S0218202518500276}
\end{barticle}
\endbibitem

%%% 21
\bibitem[\protect\citeauthoryear{Pinnau et~al.}{2017}]{pinnau2017consensus}
\begin{barticle}
\bauthor{\bsnm{Pinnau}, \binits{R.}},
\bauthor{\bsnm{Totzeck}, \binits{C.}},
\bauthor{\bsnm{Tse}, \binits{O.}},
\bauthor{\bsnm{Martin}, \binits{S.}}:
\batitle{A consensus-based model for global optimization and its mean-field
  limit}.
\bjtitle{Mathematical Models and Methods in Applied Sciences}
\bvolume{27}(\bissue{01}),
\bfpage{183}--\blpage{204}
(\byear{2017})
\doiurl{10.1142/S0218202517400061}
\end{barticle}
\endbibitem

%%% 22
\bibitem[\protect\citeauthoryear{Burger et~al.}{2014}]{bib:Burgeretal:OCPMFG}
\begin{barticle}
\bauthor{\bsnm{Burger}, \binits{M.}},
\bauthor{\bsnm{Francesco}, \binits{M.D.}},
\bauthor{\bsnm{Markowich}, \binits{P.A.}},
\bauthor{\bsnm{Wolfram}, \binits{M.-T.}}:
\batitle{Mean field games with nonlinear mobilities in pedestrian dynamics}.
\bjtitle{Discrete and Continuous Dynamical Systems - B}
\bvolume{19}(\bissue{5}),
\bfpage{1311}--\blpage{1333}
(\byear{2014})
\doiurl{10.3934/dcdsb.2014.19.1311}
\end{barticle}
\endbibitem

%%% 23
\bibitem[\protect\citeauthoryear{Toscani}{2006}]{bib:Toscani:kineticmodel}
\begin{barticle}
\bauthor{\bsnm{Toscani}, \binits{G.}}:
\batitle{Kinetic models of opinion formation}.
\bjtitle{Comm. Math. Sci.}
\bvolume{4}(\bissue{3}),
\bfpage{481}--\blpage{496}
(\byear{2006})
\doiurl{10.4310/CMS.2006.v4.n3.a1}
\end{barticle}
\endbibitem

%%% 24
\bibitem[\protect\citeauthoryear{Sznajd-Weron and
  Sznajd}{2000}]{sznajd2000opinion}
\begin{barticle}
\bauthor{\bsnm{Sznajd-Weron}, \binits{K.}},
\bauthor{\bsnm{Sznajd}, \binits{J.}}:
\batitle{Opinion evolution in closed community}.
\bjtitle{Int. J. Mod. Phys. C}
\bvolume{11}(\bissue{06}),
\bfpage{1157}--\blpage{1165}
(\byear{2000})
\doiurl{10.1142/S0129183100000936}
\end{barticle}
\endbibitem

%%% 25
\bibitem[\protect\citeauthoryear{Düring and Wright}{2022}]{bib:During:polling}
\begin{barticle}
\bauthor{\bsnm{Düring}, \binits{B.}},
\bauthor{\bsnm{Wright}, \binits{O.}}:
\batitle{On a kinetic opinion formation model for pre-election polling}.
\bjtitle{Phil. Trans. R. Soc. A.}
\bvolume{380}(\bissue{2224}),
\bfpage{20210154}
(\byear{2022})
\doiurl{10.1098/rsta.2021.0154}
\end{barticle}
\endbibitem

%%% 26
\bibitem[\protect\citeauthoryear{D{\"u}ring and Wright}{2024}]{during2024voter}
\begin{barticle}
\bauthor{\bsnm{D{\"u}ring}, \binits{B.}},
\bauthor{\bsnm{Wright}, \binits{O.}}:
\batitle{Voter demographics and socio-economic factors in kinetic models for
  opinion formation}.
\bjtitle{arXiv preprint arXiv:2412.02461}
(\byear{2024})
\doiurl{10.48550/arXiv.2412.02461}
\end{barticle}
\endbibitem

%%% 27
\bibitem[\protect\citeauthoryear{Bartel et~al.}{2024}]{bib:Bartel}
\begin{botherref}
\oauthor{\bsnm{Bartel}, \binits{H.}},
\oauthor{\bsnm{Lampert}, \binits{J.}},
\oauthor{\bsnm{Ranocha}, \binits{H.}}:
Structure‐preserving numerical methods for {F}okker–{P}lanck equations.
PAMM
\textbf{24}(4)
(2024)
\doiurl{10.1002/pamm.202400007}
\end{botherref}
\endbibitem

\end{thebibliography}

\end{document}